\documentclass[12pt,a4paper,final]{iopart}

\usepackage{graphicx}
\usepackage[breaklinks=true,colorlinks=true,linkcolor=blue,urlcolor=blue,citecolor=blue]{hyperref}
\usepackage{amsopn}
\usepackage{booktabs}
\usepackage{algorithm}
\usepackage{algpseudocode}
\usepackage{amsopn}

\RequirePackage{alltt, moreverb}
\RequirePackage{showkeys, showidx}
\RequirePackage{exscale}
\RequirePackage{tikz}
\RequirePackage{color, colortbl}
\RequirePackage{rotating}
\RequirePackage{boxedminipage}
\RequirePackage{upgreek, wasysym}
\RequirePackage{subfigure}
\RequirePackage{units}
\RequirePackage{enumitem}
\RequirePackage{amsthm}

\theoremstyle{plain}
\newtheorem{theorem}{Theorem}[section]

\newtheorem{lemma}      [theorem]{Lemma}

\renewcommand{\t} {^{\top}}

\newcommand{\norm} [2][]{\left\|#2\right\|_{#1}}

\newcommand{\true} {{\rm true}}

\newcommand{\bfepsilon}{{\boldsymbol{\epsilon}}}
\newcommand{\bfvarepsilon}{{\boldsymbol{\varepsilon}}}

\newcommand{\bfA}{{\bf A}}
\newcommand{\bfB}{{\bf B}}
\newcommand{\bfC}{{\bf C}}
\newcommand{\bfD}{{\bf D}}

\newcommand{\bfI}{{\bf I}}

\newcommand{\bfL}{{\bf L}}
\newcommand{\bfM}{{\bf M}}

\newcommand{\bfP}{{\bf P}}
\newcommand{\bfQ}{{\bf Q}}

\newcommand{\bfW}{{\bf W}}

\newcommand{\bfa}{{\bf a}}
\newcommand{\bfb}{{\bf b}}

\newcommand{\bfs}{{\bf s}}

\newcommand{\bfv}{{\bf v}}

\newcommand{\bfx}{{\bf x}}

\newcommand{\bfzero}{{\bf0}}

\newcommand{\calR}{\mathcal{R}}

\newcommand{\bbE}{\mathbb{E}}

\newcommand{\bbN}{\mathbb{N}}

\newcommand{\bbR}{\mathbb{R}}

\usepackage{amsfonts}
\usepackage{multirow}

\newcommand{\cref}[1]{\ref{#1}}

\expandafter\let\csname equation*\endcsname\relax
  \expandafter\let\csname endequation*\endcsname\relax

\RequirePackage{amsmath}
\usepackage{mathtools}
\DeclareMathOperator*{\argmin}{arg\,min}

\begin{document}

\title[Sampled Limited Memory Methods]{Sampled Limited Memory Methods for Massive Linear Inverse Problems}

\author{Julianne Chung}
\address{Department of Mathematics, Computational Modeling and Data Analytics Division, Academy of Integrated Science, Virginia Tech, Blacksburg, VA, USA}
\ead{jmchung@vt.edu}

\author{Matthias Chung}
\address{Department of Mathematics, Computational Modeling and Data Analytics Division, Academy of Integrated Science, Virginia Tech, Blacksburg, VA, USA}
\ead{mcchung@vt.edu}

\author{J. Tanner Slagel}
\address{Department of Mathematics, Virginia Tech, Blacksburg, VA, USA}
\ead{slagelj@vt.edu}

\author{Luis Tenorio}
\address{Department of Applied Mathematics and Statistics, Colorado School of Mines, Golden, CO, USA}
\ead{ltenorio@mines.edu}

\begin{abstract}
In many modern imaging applications the desire to reconstruct high resolution images, coupled with the abundance of data from acquisition using ultra-fast detectors, have led to new challenges in image reconstruction.  A main challenge is that the resulting linear inverse problems are massive. The size of the forward model matrix exceeds the storage capabilities of computer memory, or the observational dataset is enormous and not available all at once.  Row-action methods that iterate over samples of rows can be used to approximate the solution while avoiding memory and data availability constraints.  However, their overall convergence can be slow.
In this paper, we introduce a sampled \emph{limited memory} row-action method for linear least squares problems, where an approximation of the global curvature of the underlying least squares problem is used to speed up the initial convergence and to improve the accuracy of iterates. We show that this limited memory method is a generalization of the damped block Kaczmarz method, and we prove linear convergence of the expectation of the iterates and of the error norm up to a convergence horizon.
Numerical experiments demonstrate the benefits of these sampled limited memory row-action methods for massive 2D and 3D inverse problems in tomography applications.
\end{abstract}


\section{Introduction}\label{sec:intro}
Recent advancements in imaging technology have led to many new challenging mathematical problems for image reconstruction. One problem that has gained significant interest, especially in the age of big data, is that of image reconstruction when the number of unknown parameters is huge (i.e., images with very high spatial resolution) and the size of the observation dataset is massive and possibly growing \cite{parkinson2017machine,parkinson2018achieving}. In streaming data problems, not only is it undesirable to wait until all data have been collected to obtain a reconstruction, but also partial reconstructions may be needed to inform the data acquisition process (e.g., for optimal experimental design or model calibration).

We consider linear inverse problems of the form
\begin{equation}\label{eq:inverseproblem}
  \bfb = \bfA\bfx_{\rm true} +\bfvarepsilon,
\end{equation}
where $\bfx_{\rm true}\in \bbR^{n}$ is the desired solution, $\bfA \in \bbR^{m \times n}$ models the forward process, $\bfb \in \bbR^{m}$ contains observed measurements, and $\bfvarepsilon \in \bbR^m$ represents additive noise. We investigate sampled limited memory row-action methods to approximate a solution to the corresponding massive least squares (LS) problem,
$\min_{\bfx} \norm[2]{\bfA\bfx-\bfb}^{2}.$
For $\bfA$ with full column rank, the goal is to approximate the unique solution,
\begin{equation}
  \label{eq:LS}\bfx_{\text{LS}} = \argmin_{\bfx} \norm[2]{\bfA\bfx-\bfb}^{2} = (\bfA\t \bfA)^{-1} \bfA\t \bfb.
\end{equation}
The LS problem is ubiquitous and core to computational mathematics and statistics.  However, for massive problems where both the number of unknown parameters $n$ and the size of the observational dataset $m$ are massive or dynamically growing, standard techniques based on matrix factorization or iterative methods based on \emph{full} matrix-vector multiplications (e.g., Krylov subspace methods \cite{paige1982lsqr} or randomized methods \cite{avron2010blendenpik, halko2011finding}) are not feasible. Problems of such nature appear more frequently and are plentiful, for instance in classification problems \cite{bottou2004large}, data mining \cite{hand2000data,Rajaraman:2011:MMD:2124405,zeng2014incremental}, 3-D molecular structure determination \cite{boumal2018heterogeneous,levin20183d}, and super-resolution imaging \cite{chung2006numerical,slagel2019sampled}.

\emph{Row-action methods} such as the Kaczmarz method have emerged as an attractive alternative due to their scalability, simplicity, and quick initial convergence \cite{censor2009note,eldar2011acceleration,Gower2015,needell2010randomized,needell2016stochastic,Needell2014,strohmer2009comments,strohmer2009randomized}.
Furthermore, for ill-posed inverse problems such as in tomography reconstruction, row-action methods are widely used and exhibit regularizing properties \cite{andersen2014generalized,elfving2014semi}.
Basically, row-action methods are iterative methods where each iteration only requires a sample of rows of $\bfA$ and $\bfb$, thus circumventing issues with memory or data access.
The most widely-known row-action methods are the Kaczmarz and block Kaczmarz methods \cite{gordon1970algebraic}, where only one row or one block of rows of $\bfA$ and $\bfb$ are required at each iteration.  Various extensions have been proposed to include random sampling and damping parameters (e.g., \cite{elfving1980block,andersen2014generalized}), and many authors have studied the convergence properties of these methods
\cite{strohmer2009randomized,strohmer2009comments,Needell2014,needell2016stochastic}. In Section \ref{sec:ram} we provide a detailed review of previous works in this area, but first we present the problem setup.

To mathematically describe the sampling process, let $\{\bfW_k\}$ be an independent and identically distributed (i.i.d.) sequence of $m\times \ell$ random matrices uniformly distributed on the set $\left\{\bfW^{(1)},\ldots, \bfW^{(M)}\right\}$, where $\bfW^{(i)}$ are such that $\bbE \bfW_k\bfW_k\t = \beta \bfI$ for some $\beta>0$.
Then, at the $k$-th iteration we denote
$\bfA_{k}=\bfW\t_{k}\bfA$ and $\bfb_{k}= \bfW\t_{k}\bfb$.
In this paper, we focus on a row-action method called \emph{sampled limited memory for LS} (\texttt{slimLS}), where given an arbitrary initial guess $\bfx_{0} \in \mathbb{R}^{n}$,
 the $k$-th \texttt{slimLS} iterate is defined as
 \begin{align} \label{eq:ram} \bfx_{k} = \bfx_{k-1} - \bfB_{k}\bfA\t_{k}\left(\bfA_{k}\bfx_{k-1}-\bfb_{k}\right), \end{align}
  with
  \begin{equation}
    \label{eq:B_k}
  \bfB_{k}  = \left(\alpha^{-1}_{k}\bfC_{k} + \bfM\t_{k}\bfM_{k}\right)^{-1} \quad\mbox{and}\quad
    \bfM_{k} = [\bfA_{k-r}\t, \ldots,  \bfA_{k}\t ]\t.  \end{equation}
  Here $\{\bfC_{k}\}$ is a sequence of positive definite matrices, $\{\alpha_k\}$
 is sequence of damping parameters,
and the parameter $r\in\bbN_0$ is a ``memory parameter'' where we define $\bfA_{k-r}$ with negative index as an empty matrix. Hence, $\bfM_k$ increases in size within the first $r$ iterations, and the global curvature of the problem is approximated using previous samples of the data.

Various choices for $\bfW$ can be used, e.g., see \cite{chung2017stochastic}.  Notice that when the realizations of $\bfW$ are sparse matrices, only information from a few rows of $\bfA$ is extracted at each iteration, and $\bfA_k$ contains partial information of $\bfA$.
A straightforward choice of $\bfW$, which we consider here, is where $\bfW_k$ is chosen such that $\bfA_k$ contains $\ell$ selected rows of $\bfA$ where $\beta = 1/M$.
With this sampling scheme, the \texttt{slimLS} method is a generalization of the damped block Kaczmarz method, which is obtained when $\bfC_{k} = \bfI$ and $r=0$.
The \texttt{slimLS} method can also be interpreted as a stochastic approximation method \cite{shapiro2009lectures}. Using the properties of $\bfW$ described above, one can show that the LS problem in \eqref{eq:LS} is equivalent to the following stochastic optimization problem,
 \begin{equation}
 \label{eq:so}
  \argmin_\bfx \ \mathbb{E} \norm[2]{\bfW\t\left(\bfA\bfx-\bfb\right)}^{2}.
 \end{equation}
 Stochastic approximation methods for \eqref{eq:so} may have the form
   $\bfx_{k} = \bfx_{k-1} - \bfB_{k} \nabla f_{\bfW_{k}}\left(\bfx_{k-1}\right),$
 where $f_{\bfW_{k}}\left(\bfx\right) = \norm[2]{\bfW_{k}\t\left(\bfA\bfx-\bfb\right)}^{2}$ and $\{\bfB_{k} \}$ is some sequence of positive semi-definite matrices.
 For the particular choice of $\bfB_k$ defined in \eqref{eq:B_k}, we see that the \texttt{slimLS} method is a stochastic approximation method.
 Furthermore, since samples of rows of $\bfA$ are used at each iteration, randomized row-action methods can be characterized as stochastic approximation methods applied to \eqref{eq:so}.

Although our proposed \texttt{slimLS} method can be interpreted as both a row-action and a stochastic approximation method, the main distinction of \texttt{slimLS} compared to existing methods to approximate~\eqref{eq:LS} is that \texttt{slimLS} exhibits favorable initial \emph{and} asymptotic convergence properties for constant and vanishing step sizes, respectively.  Kaczmarz methods have fast initial convergence, but for inconsistent systems iterates converge asymptotically to a weighted LS solution rather than the desired LS solution \cite{needell2010randomized}. On the other hand, stochastic gradient methods (where $\bfB_k = \alpha_k\bfI$) are guaranteed to converge asymptotically to the LS solution but can have erratic initial convergence.
We show linear convergence of the expectation of \texttt{slimLS} iterates, and we prove linear convergence of the expected squared error norm up to a ``convergence horizon" for constant damping parameter. Furthermore, it can be shown that \texttt{slimLS} iterates with decaying damping parameter converge asymptotically to the LS solution \cite{chung2017stochastic,Bottou1998}.
The power of the \texttt{slimLS} method is revealed in our numerical examples, where we demonstrate the performance of the \texttt{slimLS} method for massive and streaming tomographic reconstruction problems.

 An outline of the paper is as follows. In Section \ref{sec:ram}, we give an overview of previous work on row-action methods and make connections to and distinctions from existing methods. In Section \ref{sec:conv}, we provide convergence results for \texttt{slimLS} iterates.
 Numerical results are presented in Section \ref{sec:num}, where we compare the performance of \texttt{slimLS} to existing methods, and conclusions are provided in Section \ref{sec:con}.

 \section{Previous works on row-action methods} \label{sec:ram}
 Different choices of $\bfB_{k}$ in \eqref{eq:ram} yield different well-known row-action methods.
 The most computationally simple choice is $\bfB_{k} = \alpha_{k}\bfI$, which gives the standard stochastic gradient method,
 \begin{align*}
 \bfx_{k} = \bfx_{k-1} - \alpha_{k} \bfA\t_{k}\left(\bfA_{k}\bfx_{k-1}-\bfb_{k}\right).
 \end{align*}
 Under mild conditions, the stochastic gradient method converges to the LS solution \cite{Bottou1998,chung2017stochastic}, but convergence can be very slow and depends heavily on the choice of the step size $\alpha_k$ \cite{schaul2013no,zeiler2012adadelta}.
 We remark that $\alpha_k$ has different interpretations depending on the scientific community.  It is often referred to as the learning rate in machine learning, the step size in classical optimization, and the relaxation parameter in algebraic reconstruction techniques for tomography. Notice that for \texttt{slimLS} iterates, the damping parameter plays the role of the step size.

 Stochastic Newton or stochastic quasi-Newton methods have also been proposed to accelerate convergence \cite{chung2017stochastic,bottou2004large,gower2017randomized,Byrd2016}. For the stochastic Newton method, we can let $\bfB_{k} = \alpha_k\left(\bfA\t_{k}\bfA_{k}\right)^{\dagger}$ in \eqref{eq:ram}, and we get the block Kaczmarz method,
 \begin{align} \label{eq:stochnewt1}
   \begin{split} \bfx_{k}  & = \bfx_{k-1} - \alpha_{k}\left(\bfA\t_{k}\bfA_{k}\right)^{\dagger}\bfA_{k}\t\left(\bfA_{k}\bfx_{k-1}-\bfb_{k}\right)\\
    & = \bfx_{k-1} - \alpha_{k}\bfA_{k}^{\dagger}\left(\bfA_{k}\bfx_{k-1}-\bfb_{k}\right),
  \end{split}
 \end{align}
 where we have used the property that $\left(\bfA\t_{k}\bfA_{k}\right)^{\dagger}\bfA_{k}\t = \bfA_{k}^{\dagger}$.
 Thus, we see that the block Kaczmarz method is nothing more than a stochastic Newton method. Note that for $\alpha_k=1$, we get the standard block Kaczmarz method, and linear convergence to within a convergence horizon has been shown in \cite{Needell2014,needell2015randomized}.
 For a decaying $\alpha_k$, the block Kaczmarz method converges to the solution of a weighted LS problem, rather than to the desired LS solution \cite{censor1983strong}.

 For the special case where $\bfW$ is a uniform random vector on the columns of the identity matrix, each iteration only requires one row of $\bfA$.  More precisely, let $\bfa_{i} \in \mathbb{R}^{1 \times n}$ denote the $i$-th row of $\bfA$ and let $\tau$ be a random variable with uniform distribution on the set $\{1, \dots, m\}$, then $\bfW_k\t\bfA = \bfa_{\tau(k)}$.
 In this case, stochastic Newton iterates in \eqref{eq:stochnewt1} are identical to the randomized Kaczmarz method,
 \begin{align}  \label{eq:Kacz} \bfx_{k} = \bfx_{k-1} -  \alpha_{k} \frac{\bfa_{\tau(k)}\bfx_{k-1} - b_{\tau(k)}}{\|\bfa_{\tau(k)}\t\|_2^{2}}\bfa\t_{\tau(k)}, \end{align}
 which has been studied extensively, see e.g., \cite{strohmer2009randomized,feichtinger1992new,herman1993algebraic,natterer2001mathematics,whitney1967two,censor1983strong,tanabe1971projection,hanke1990acceleration,zouzias2013randomized,needell2010randomized}.
 For consistent systems where $\bfA$ has full column rank, it was first shown in $1937$ that the Kaczmarz method with cyclic control (i.e., $\tau(k) = ((k-1) \,\,\text{mod}\,\, m)+1$) converges to the LS solution \cite{kaczmarz1937angeniherte}.  The method has gained widespread popularity in tomography applications, where it is commonly known as the Algebraic Reconstruction Technique \cite{natterer2001mathematics, censor1983strong, tanabe1971projection,hanke1990acceleration,gordon1970algebraic,herman2009fundamentals}. The randomized Kaczmarz method was shown to have an expected linear convergence rate that depends on the condition number of $\bfA$ \cite{strohmer2009randomized,strohmer2009comments}. For inconsistent systems the Kacmarz method does not converge to the LS solution. For a decaying step size $\alpha_k$, iterates will converge to a weighted LS solution $\tilde{\bfx} = \argmin_\bfx \norm[2]{\bfD^{-1}\left(\bfA\bfx-\bfb\right)}^{2}$  where $\bfD \in \mathbb{R}^{m \times m}$ is a diagonal matrix with diagonal elements $d_{i} = \norm[2]{\bfa_{i}}^{2}$, see \cite{censor1983strong,chung2017stochastic}. For a constant step size, these iterates will converge linearly to the weighted LS solution to within what is known as a convergence horizon, which accounts for the variance in the iterates \cite{needell2010randomized,needell2016stochastic}.

 To address problems that arise when some rows have small norm (i.e., a small denominator in~\eqref{eq:Kacz}), Andersen and Hansen \cite{andersen2014generalized} in 2014 considered a variant of the Kaczmarz method to include a \emph{damping} term. They showed a connection to proximal gradient methods and provided convergence properties under cyclic control.
  When the blocks $\bfA_{k}$ are ill-conditioned, computing the search direction in \eqref{eq:stochnewt1} can become numerically unstable and a similar idea can be used. A damping term can be introduced in the sample Hessian, which leads to the damped block Kaczmarz method,
  \begin{equation}
    \label{eq:dbK}
     \bfx_{k} =  \bfx_{k-1} - \left(\alpha^{-1} _{k}\bfI + \bfA\t_{k}\bfA_{k}\right)^{-1}\bfA\t_{k}\left(\bfA_{k}\bfx_{k}-\bfb_{k}\right).
  \end{equation}
 Notice that including the damping parameter eliminates the need for a step size parameter, although one could still be included.

 To speed up convergence, stochastic quasi-Newton methods use the current and any previous samples of $\bfA$ to produce a matrix $\bfB_{k}$ that approximates the global curvature $\left(\bfA\t\bfA\right)^{-1}$. For general convex optimization, stochastic quasi-Newton methods that use an LBFGS type update have been introduced and analyzed \cite{Byrd2016,Gower2016s,mokhtari2015global}. These methods have been investigated for nonlinear problems; however, for linear problems better approximations can be obtained by exploiting the fact that the Hessian is constant.
 One row-action method for linear problems is the randomized recursive LS method where the $k$-th iterate, which is given by
 \begin{equation}
   \label{eq:rrls}
    \bfx_{k} = \bfx_{k-1} - \left(\sum_{i=1}^{k}\bfA\t_{i}\bfA_{i}\right)^{\dagger}\bfA\t_{k}\left(\bfA_{k}\bfx_{k}-\bfb_{k}\right),
 \end{equation}
 is the minimum norm solution of
 \begin{equation}
   \label{eq:rrls2}
    \min_{\bfx} \norm[2]{\begin{bmatrix} \bfA_{1} \\ \vdots \\ \bfA_{k} \end{bmatrix}\bfx - \begin{bmatrix} \bfb_{1} \\ \vdots \\ \bfb_{k} \end{bmatrix} }^{2}.
 \end{equation}
 This equivalency is shown in \ref{sec:proofs2} and implies that after sampling all $M$ blocks exactly once, we get $\bfx_{M} = \bfx_{\text{LS}}$. The disadvantage is that
 the randomized recursive LS algorithm is not computationally feasible for very large problems because of the large linear solve required at each iteration and the cost to store $\sum_{i=1}^{k}\bfA\t_{i}\bfA_{i}$, see \cite{chen2006regression,kushner2003stochastic,slagel2019sampled}.
 Notice that if $\bfC_{k} = \sum_{i=1}^{k-r-1}\bfA\t_{i}\bfA_{i}$ and $\alpha_{k}=1$ in \eqref{eq:ram}, we recover the randomized recursive LS algorithm. Thus, the \texttt{slimLS} iterates \eqref{eq:ram} can be interpreted as a limited memory variant of the recursive LS algorithm.  On the other hand, if $\bfC_{k}=\alpha^{-1}_{k}\bfI_{n}$ the \texttt{slimLS} method can be interpreted as a generalization of the damped block Kaczmarz method.

 It should be noted that other methods exist for solving very large LS problems, but many have limitations that prohibit their use for massive or streaming data problems. For example, for problems where $m$ is small enough to allow storage of an $m \times m$  matrix, the Sherman-Woodbury identity can be used to get the exact solution \cite{egidi2006sherman}.  In our problems, $m$ and $n$ are on the order of hundreds of millions.  Randomized methods such as Blendenpik \cite{avron2010blendenpik} and LSRN \cite{meng2014lsrn} are effective for cases where $m \gg n $ or $n \gg m$, (assuming matrix $\bfA$ has a large gap in the singular values). Nevertheless, these methods require full matrix-vector-multiplications and do not work for streaming problems.

 Next, for a specific choice of $\bfW_k$, we make a connection to the sampled limited memory Tikhonov (\texttt{slimTik}) method described in \cite{slagel2019sampled} to approximate a Tikhonov regularized solution,
 \begin{equation}\label{eq:TLS1}
    \bfx_{\text{tik}} = \argmin_{\bfx} \ \ \norm[2]{\bfA\bfx-\bfb}^{2} + \lambda^{2} \norm[2]{\bfL\bfx}^{2}\\
    = \norm[2]{ \left[\begin{array}{c} \bfA \\ \lambda \bfL \end{array}\right] \bfx - \left[\begin{array}{c} \bfb \\ \bf0 \end{array}\right]}^{2},
 \end{equation}
 where $\lambda\in\bbR^+$ is the regularization parameter and the regularization matrix $\bfL\in \bbR^{n\times n}$ is invertible\footnote{This assumption is not required \cite{hansen2010discrete} but is used here for notational simplicity.}.
 In particular, let $\bfW_k$ be defined as in Section \ref{sec:intro} and define random variable
 \begin{equation*}\widetilde{\bfW}_k = \left[\begin{array}{cc} {\bfW_k} & {\bf0}_{m \times n} \\ {\bf0}_{n \times \ell} & \frac{1}{\sqrt{M}}\bfI \end{array}\right], \end{equation*}
 where $\widetilde{\bfW}_k$ has the property that $\mathbb{E}\widetilde{\bfW}_k {\widetilde{\bfW}_k}\t = \tfrac{1}{M}\bfI$.
 Then, \texttt{slimLS} applied to \eqref{eq:TLS1} with $\bfC_{k} = \bfL\t\bfL$
 gives iterates,
 \begin{align}
   \label{eq:slimTik}
 \bfx_{k} & = \bfx_{k-1} - \left(\left(\alpha^{-1}_{k} +\frac{r \lambda^2}{M}\right)\bfL\t\bfL  + \bfM\t_{k}\bfM_{k}\right)^{-1}\left(\bfA\t_{k}\left(\bfA_{k}\bfx_{k-1}-\bfb_{k}\right) + \frac{\lambda^2}{M}\bfL\t\bfL\bfx_{k-1}\right),
 \end{align}
 which are equivalent to \texttt{slimTik} iterates with a fixed regularization parameter.
 The significance of this equivalence is that the analysis and results that we present in the next section can be extended to the Tikhonov problem.
 It should be noted that a good choice of $\lambda$ may not be known in advance and must be estimated.  Methods to update the Tikhonov regularization parameter within the \texttt{slimTik} method have been studied in \cite{slagel2019sampled}, but a theoretical analysis for such cases remains a topic of ongoing research.

 \section{Convergence properties of \texttt{slimLS}}  \label{sec:conv}
 In this section we analyze the convergence properties of the \texttt{slimLS} method.  In particular, we will show that it exhibits favorable initial convergence properties without the memory burden of having to save all previous samples (e.g., as in randomized recursive LS \cite{bjorck1996numerical,slagel2019sampled}).  This is possible because \texttt{slimLS} iterates can utilize previous samples to better approximate the Hessian $\bfA\t\bfA$.
 We show that for a fixed damping parameter $\alpha >0$, memory level $r=0$, and $\bfC_{k} = \bfI$,
 \texttt{slimLS} iterates exhibit linear convergence of the expectation of the iterates and linear convergence of the $L^2$-error up to a convergence horizon.
 This type of analysis is essential for understanding stochastic approximation methods \cite{kushner2003stochastic,Needell2014,needell2016stochastic,strohmer2009randomized,Gower2015}, and it may reveal potential trade-offs between solution accuracy and speed of convergence based on the damping parameter.
 For example,
 such analyses have been proved for the Kaczmarz method (for vanishing step size) and for the block Kaczmarz method (for step size one) \cite{Gower2015,needell2016stochastic,Needell2014,strohmer2009randomized}, but to the best of our knowledge results have not been proved for the damped block Kaczmarz method.
 For clarity of presentation, all proofs have been relegated to \ref{sec:proofs}.

 \medskip
 The following definitions will be used throughout the paper.
 We will use the functions $\lambda_{\min}(\cdot)$ and $\lambda_{\max}(\cdot)$ that provide the smallest and largest
 eigenvalues of a matrix, and write
 \[
 A_{\min} = \min_k \left\{\lambda_{\min}(\bfA\t\bfW^{(i)}\left(\bfW^{(i)}\right)\t\bfA) >0 \right\}  \mbox{ and } A_{\max} = \max_k\lambda_{\max}(\bfA\t\bfW^{(i)}\left(\bfW^{(i)}\right)\t\bfA),
 \]
 where the minimum is across all of the $M$ different realizations of $\bfW_k$ that lead to a positive minimum eigenvalue, while the maximum is across all of the $M$ realizations.
 For a fixed $\alpha>0$ we define the matrices
 \[
 \bfB_k(\alpha)=\alpha \left(\bfI+\alpha \,\bfA\t_{k}\bfA_{k}\right)^{-1}\quad \mbox{and} \quad \bfB = \mathbb{E}\,\bfB_k(\alpha)\bfA\t_{k}\bfA_{k}
 = \bfI -\mathbb{E}\,\bfB_k(\alpha)/\alpha.
 \]
 For simplicity we will often write $\bfB_k$ instead of $\bfB_k(\alpha)$.
 It is clear that $\bfB$ is symmetric positive semi-definite. In fact, it is positive definite with $\norm[2]{\bfB} < 1$ when $\bfA$ has full column-rank (see \ref{lem:B}), in which case we define
 \begin{equation}
   \label{eq:xhat}
 \widehat{\bfx} = \argmin_{\bfx} \norm[2]{\bfB\bfx-\mathbb{E}\,\bfB_{k}\bfA\t_{k}\bfb_{k}}^{2} \\
  = \bfB^{-1}\mathbb{E}\,\bfB_{k}\bfA\t_{k}\bfb_{k}.
 \end{equation}
 Note that all of the expectations in this paper are understood to be with respect to the joint distribution of the sequence $\{\bfW_k\}$ conditional on the noise.

 \begin{theorem} \label{thm:firstmoment}
 Let $\bfA \in \mathbb{R}^{m \times n}$ have full column-rank.
 For arbitrary initial vector $\bfx_0 \in \bbR^n$ and damping parameter $\alpha >0$, define
 \begin{equation*}
 \bfx_{k}  = \bfx_{k-1} - \bfB_{k}(\alpha)\,\bfA\t_{k}\left(\bfA_{k}\bfx_{k-1}-\bfb_{k}\right)
 \end{equation*}
 Then:
   \begin{enumerate}
 \vspace{-.2cm}
     \item $\mathbb{E}\,\bfx_{k} \rightarrow \widehat{\bfx}$, or more precisely,
 \begin{equation*} \norm[2]{\mathbb{E}\, \bfx_{k} - \widehat{\bfx}}  \leq  \rho^{k}\norm[2]{\bfx_{0}-\widehat{\bfx}},\end{equation*}
 where $\rho = \norm[2]{\mathbb{E}\,\bfB_k(\alpha)/\alpha} < 1$.
 \item The $L^2$-error around $\widehat{\bfx}$ can be bounded by
 \begin{equation}
   \label{eq:horizon}
 \mathbb{E}\norm[2]{\bfx_{k}-\widehat{\bfx}}^{2}  \leq  \left(1 - 2c\right)^{k}\norm[2]{\bfx_{0}-\widehat{\bfx}}^{2} + \alpha^{2}c^{-1}\sigma^{2},
 \end{equation}
  where
  $0 < 1-2c < 1$, with $c = \lambda_{\text{\emph{min}}}\left(\bfB\right)/(1+\alpha A_{\text{\emph{max}}})$
 and $\sigma = \mathbb{E}\norm[2]{\bfA\t_{k}(\bfA_{k}\widehat{\bfx}-\bfb_{k})}$.
 \end{enumerate}
 \end{theorem}

 \medskip
 The first part of the theorem shows that as $k\to \infty$,  $\bfx_k$ is an asymptotically unbiased estimator of $\widehat{\bfx}$ with a linear convergence rate.
 The second part shows linear convergence of the $L^2$-error up to a convergence horizon.  For the case where $\alpha \to 0$, the constant in the first term of \eqref{eq:horizon} approaches one,  indicating a slowing linear convergence rate, while the second term in \eqref{eq:horizon} goes to zero, i.e., the convergence horizon gets smaller. This is because $\alpha^2/\lambda_{\text{min}}\left(\bfB\right) \rightarrow 0$ as $\alpha \rightarrow 0$.

 Having shown that $\bfx_k$ converges to $\widehat{\bfx}$ in $L^2$ as $k\to\infty$ and $\alpha\to 0$, the next question is how much $\widehat{\bfx}$ differs from the LS solution $\bfx_{\text{LS}}$. To answer this question we let
 $\bfP_{\bfA} = \bfA(\bfA^\top\bfA)^{-1}\bfA^\top$ and $\bfQ_{\bfA} = \bfI-\bfP_{\bfA}$ be, respectively, the orthogonal projections onto the column space of $\bfA$ and
 its orthogonal complement. We then have the following equivalent definition of $\widehat{\bfx}$:
 \begin{eqnarray}
 \widehat{\bfx} &=& \argmin_\bfx \norm[2]{\mathbb{E}\,\bfB_{k}\bfA\t_{k}\bfA_{k}\left(\bfx-\bfx_{\text{LS}})-
 \mathbb{E}\,\bfB_{k}\bfA\t_{k}\bfW_{k}^\top\bfQ_{\bfA}\bfb\right)}^{2}\nonumber\\
 &=& \bfx_{\text{LS}} + \bfB^{-1}\,(\,\mathbb{E}\,\bfB_{k}\bfA\t_{k}\bfW_{k}^\top\,)\,\bfQ_{\bfA}\bfb.\label{eq:xh-xls}
 \end{eqnarray}
 In particular, $\widehat{\bfx} = \bfx_{\text{LS}}$ when $\bfb$ belongs to the column space of $\bfA$. The following result
 provides a bound for $\norm[2]{\widehat{\bfx}-\bfx_{\rm LS}}$.

 \begin{lemma}\label{thm:bias}
 If $\bfA \in \mathbb{R}^{m \times n}$ has full column-rank, then
 \begin{equation*}
     \norm[2]{\widehat{\bfx}-\bfx_{\rm LS}} \leq \alpha\,\frac{M(1+\alpha A_{\min})A_{\max}}{(1+\alpha A_{\max})A_{\min}}  \, C\,\|\bfQ_{\bfA}\bfb\|_2,
 \end{equation*}
 where $C = \mathbb{E}\,\|\bfA_k^\top\bfW_k\|_2 + \|(\bfA^\top\bfA)^{-1}\|_2\,\|\bfA\|_2\,\mathbb{E}\,\|\bfA_k^\top\bfA_k\|_2$.
 \end{lemma}

 \medskip
 It  is important to notice the relationship between the damping parameter $\alpha$ and the
 upper bound in Lemma \ref{thm:bias}. The bound is smaller for  smaller values of $\alpha$, which makes sense in light of the asymptotic property that $\bfx_{k} \rightarrow \bfx_{\text{LS}}$ a.s.~for a decaying damping parameter $\alpha$
 (see \cite{chung2017stochastic}).  However, there is a trade-off between the convergence rate and the precision of iterates (i.e., the bias) that depends on $\alpha$:
  as $\alpha \to 0$, we get more accurate approximations, i.e., $\widehat\bfx \to \bfx_{\text{LS}}$ in $L^2$.
 On the other hand, for larger  $\alpha$ the convergence is faster at the cost of a larger convergence horizon.

 In summary, we have shown that with a fixed damping parameter, the \texttt{slimLS} iterates will converge in $L^2$ linearly to within a horizon of $\widehat \bfx$, and the expected value of the iterates converges to $\widehat \bfx$. A pictorial illustration of this convergence behavior is provided in Figure \ref{fig:horizon}.
 \begin{figure}[bt]
 \begin{center}
   \includegraphics[width=.9\textwidth]{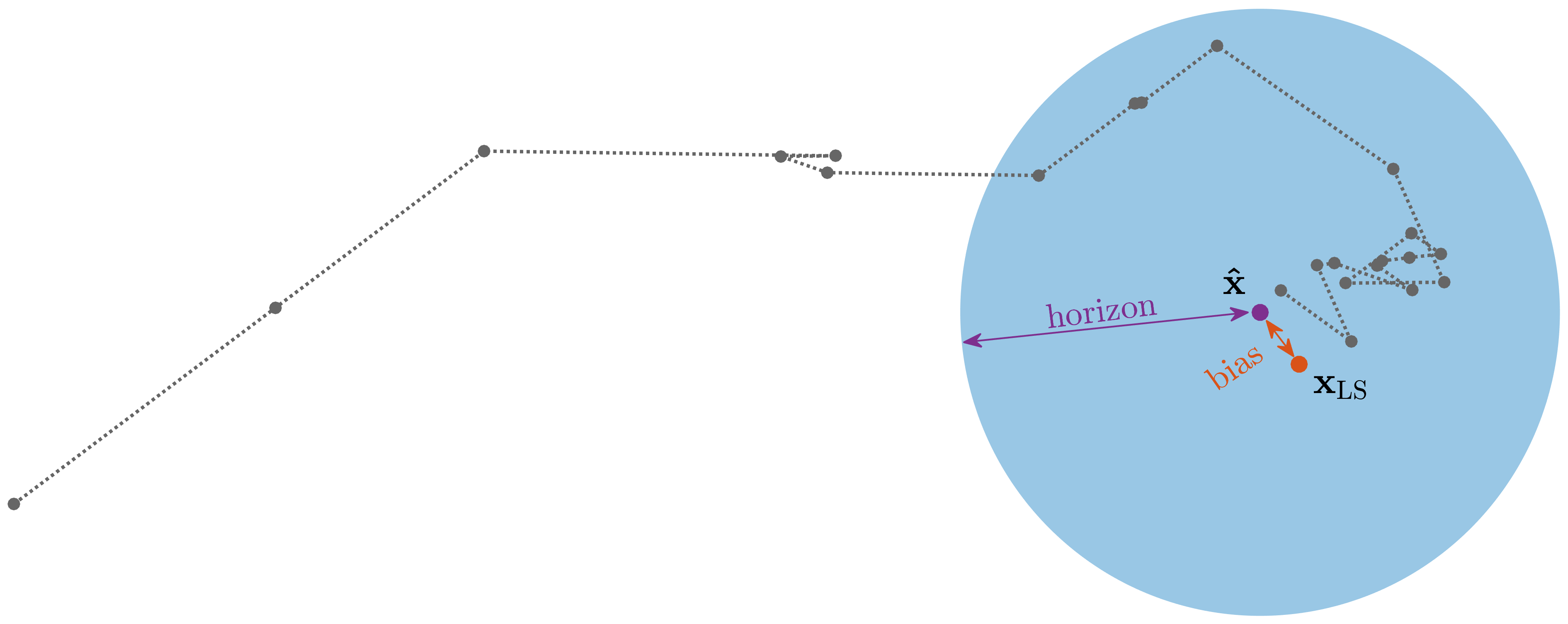}
 \caption{Illustration of the convergence horizon for \texttt{slimLS}. The gray dotted line contains the iterates $\bfx_k$ and the convergence horizon is denoted by the blue disk. By Theorem \ref{thm:firstmoment}, the \texttt{slimLS} iterates converge to within a convergence horizon of $\widehat\bfx$. $\bfx_{\rm LS}$ is provided for comparison with $\widehat\bfx$ (see Lemma \ref{thm:bias}).}
 \label{fig:horizon}
 \end{center}
 \end{figure}

 This trade-off between quick initial convergence that comes with using a constant damping parameter at the cost of solution accuracy has been observed in related stochastic optimization methods in the literature, see e.g., \cite{benveniste2012, andersen2014generalized}.
 It has also been observed that more accurate solutions can be obtained using a decaying damping parameter, but then convergence can be quite slow (sub-linear) \cite{bottou2005line, nemirovski2009robust}. Thus, it is often practical to use a constant damping parameter to obtain quick initial convergence and then switch to a decaying parameter to obtain higher accuracy.

 \section{Numerical results}\label{sec:num}
 In this section we first illustrate the convergence behavior of our proposed \texttt{slimLS} method in a small simulation study. The goal of the first experiment is to illustrate how different memory levels and damping parameters affect the convergence of \texttt{slimLS}. We also compare \texttt{slimLS} to existing row-action methods and provide a numerical investigation into the sensitivity towards the damping parameter/step size.
 Then we discuss some of the computational considerations when solving massive problems and investigate the performance of our method on very large tomographic reconstruction problems.

 \subsection{An Illustrative Example}\label{sec:num1}
 In the first numerical experiment we use a smaller example to illustrate some of the key features of the \texttt{slimLS} method. We let $\bfA \in \bbR^{1,000\times 100}$ have random entries from a standard normal distribution. We further assume that $\bfx_{\rm true} = [1,\ldots,1]\t$ and the simulated observations $\bfb$ are generated as in~\eqref{eq:inverseproblem}
 where $\bfepsilon$ is white noise with noise level $1\%$; that is, $\norm[2]{\bfepsilon}/\norm[2]{\bfA \bfx_{\true}} = 0.01$. The matrix $\bfA$ is assumed to be sampled in $M = 100$ blocks, each with block size $\ell = 10$, which corresponds to sampling matrices
 $\bfW^{(i)} = \left[\bfzero_{100\times 10(i-1)}, \bfI, \bfzero_{100\times 10(100-i)}\right]\t.$

 First we illustrate the convergence behavior investigated in Theorem \ref{thm:firstmoment} for different constant damping parameters $\alpha$.
 In Figure \ref{fig:stepsize} we provide relative reconstruction errors computed as $\norm[2]{\bfx_k - \widehat\bfx}/\norm[2]{\widehat\bfx}$ for various damping factors from $\alpha = 0.001$ to $\alpha = 10$ on a log-log scale.  We repeatedly run \texttt{slimLS} with random sampling for $100$ epochs and with memory level $r=0$.
 We observe that larger values of $\alpha$ have favorable initial convergence, but then stabilize at a larger relative error.  On the other hand, smaller values of $\alpha$ have a slower initial convergence, but to a smaller relative error.  This illustrates the trade-off between fast initial convergence and solution accuracy, as discussed in Section~\ref{sec:conv}.
 \begin{figure}[b]
   \begin{center}
     \includegraphics[width=.9\textwidth]{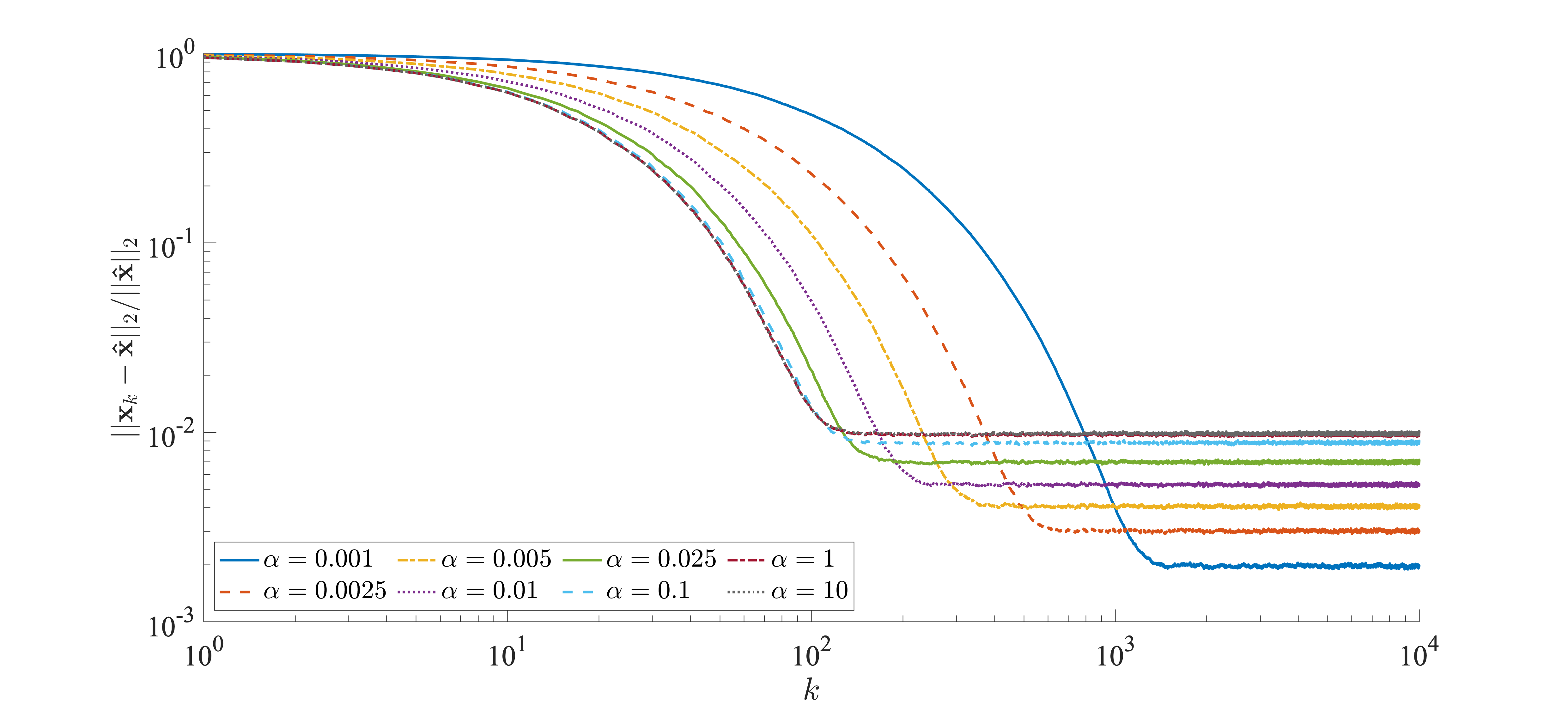}
     \caption{Comparison of median reconstruction errors, over $100$ runs and relative to $\widehat{\bfx}$, for different (fixed) damping parameters $\alpha$ in \texttt{slimLS} on a log-log scale.}\label{fig:stepsize}
   \end{center}
 \end{figure}
 Furthermore, for various $\alpha$ we provide the relative difference between $\widehat \bfx$ and $\bfx_{\rm LS}$ in Figure~\ref{fig:bound}. We notice that for small $\alpha$ the relative difference is within machine precision while slowly increasing for $\alpha> 10^{-1}$.
 \begin{figure}[bt]
   \begin{center}
 \includegraphics[width=.9\textwidth]{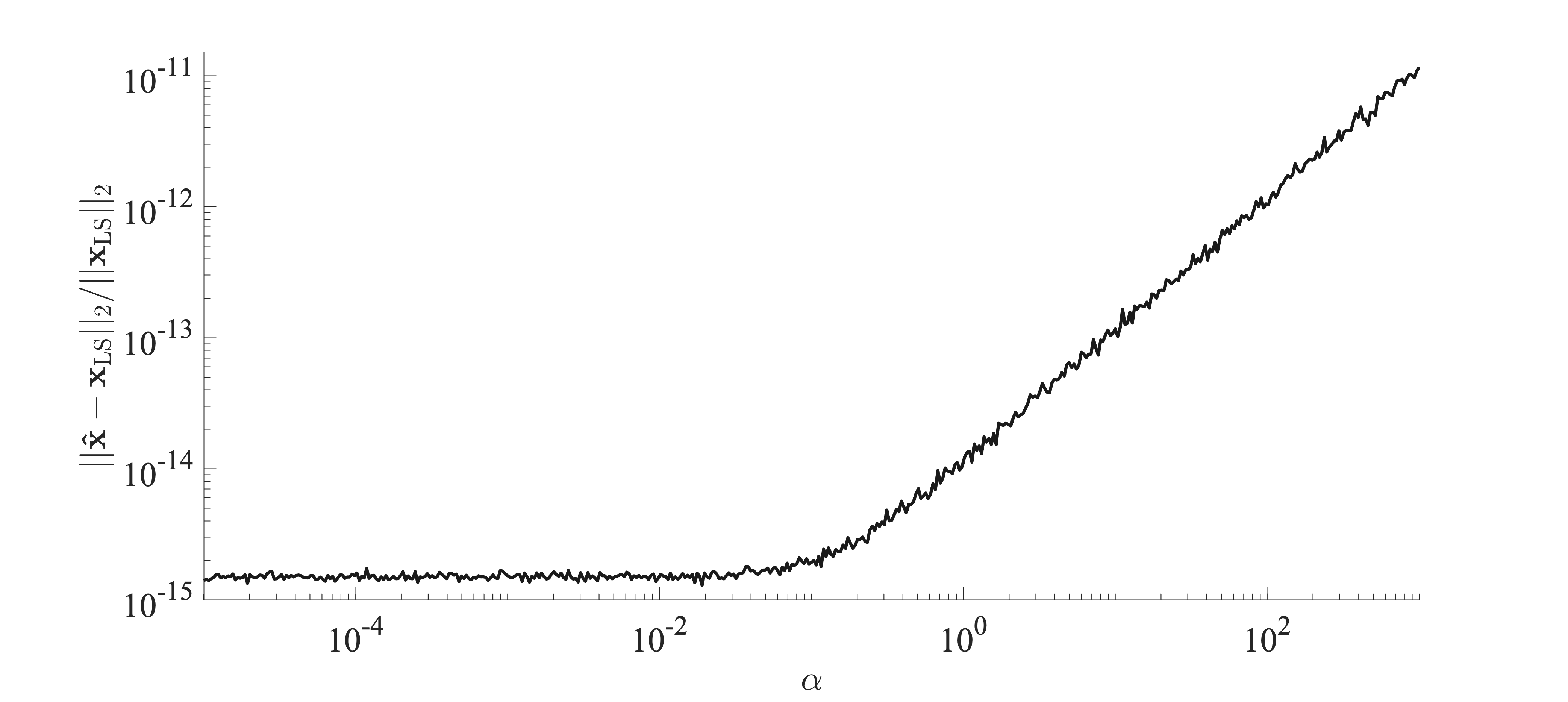}
     \caption{Relative difference between the weighted LS solution $\widehat\bfx$ and the desired LS solution $\bfx_{\rm LS}$ for various $\alpha$.\label{fig:bound}
     }
   \end{center}
 \end{figure}

 Next, we investigate how the initial convergence is affected by the choice of the memory parameter $r$. Here, we fix $\alpha=1$ and choose memory levels $r = 0, 2, 4, 6$, and $8$. We run our \texttt{slimLS} method for $20$ iterations and provide the median relative reconstruction errors for $100$ repeated runs in Figure~\ref{fig:memory}. The errors are computed with respect to the LS solution, i.e., $\norm[2]{\bfx_k - \bfx_\text{LS}}/\norm[2]{\bfx_\text{LS}}$.
 We empirically observe that with higher memory levels we get faster initial convergence.
 \begin{figure}[bt]
 \begin{center}
 \includegraphics[width=.9\textwidth]{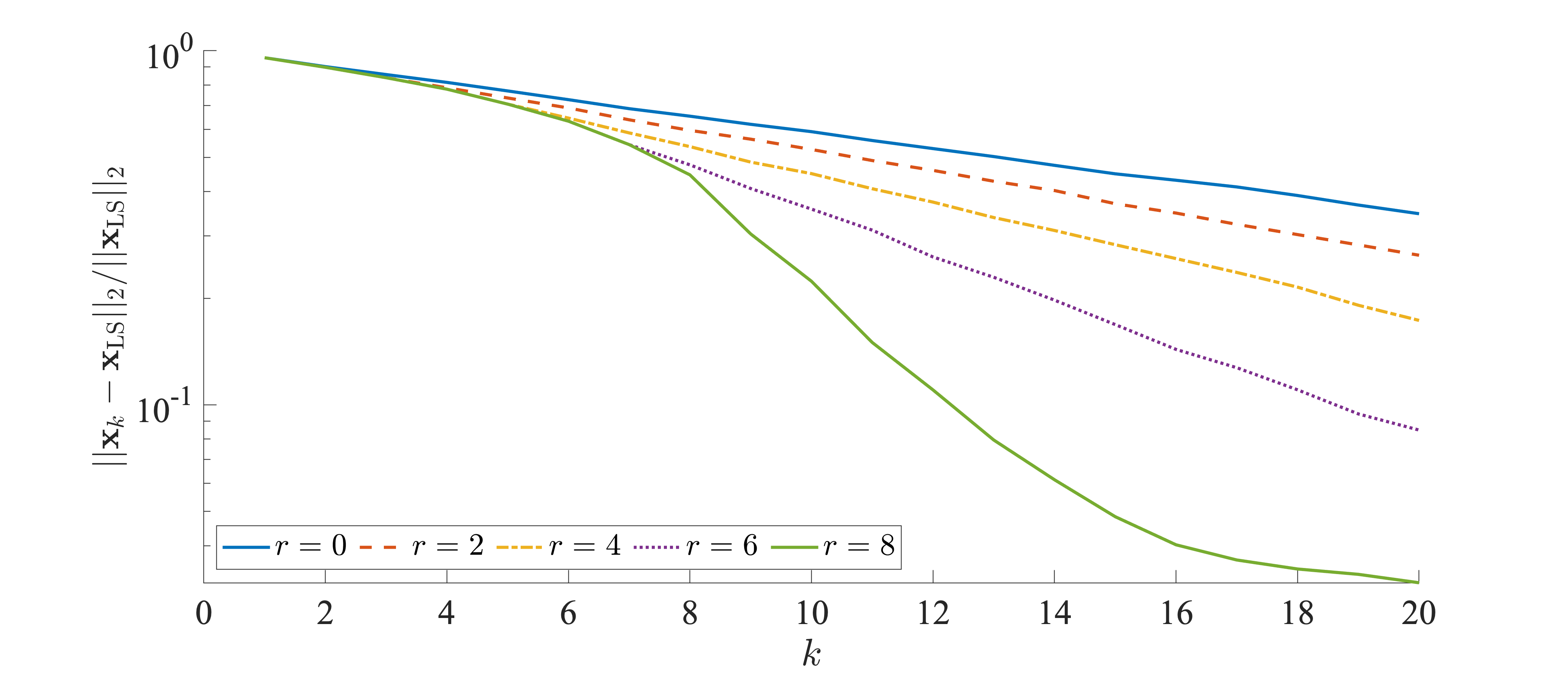}
 \caption{Comparison of median reconstruction errors, over $100$ runs and relative to $\bfx_{\rm LS}$, for different memory levels in \texttt{slimLS} with $\alpha=1$. }\label{fig:memory}
 \end{center}

 \end{figure}

 We also provide a comparison of \texttt{slimLS} with $r=0$ to other row-action methods, including the sampled or batch gradient \texttt{sg} method and the online limited memory BFGS \texttt{olbfgs} method \cite{mokhtari2015global} with memory level $10$ (i.e., storing the $10$ previously computed sampled gradient vectors).
 In particular, we are interested in the sensitivity of the algorithms with respect to the step size.
 For different constant step sizes/damping factors from $\alpha = 10^{-5}$ to $\alpha = 10^{3}$, we computed the reconstruction error norm at $k = 100$ (i.e., corresponding to one epoch) relative to $\widehat\bfx$.
 Although reconstruction errors could be computed relative to $\bfx_{\rm LS}$, we note that the relative error between $\bfx_{\rm LS}$ and $\widehat\bfx$ is negligible (see~Figure~\ref{fig:bound}).
 In Figure~\ref{fig:compareStepSizing} we provide the median relative reconstruction error norms along with the $5-95$th percentiles after repeating the experiment $100$ times. Notice that the plot is on a log-log scale.
  We observe that the \texttt{sg} method has just a tiny ``window'' of $\alpha$'s for which results have small relative reconstruction errors. The window for \texttt{olbfgs} is larger and is centered around step size $\alpha = 1$, which is expected, see \cite{bottou2018optimization}. However, compared to \texttt{sg} and \texttt{olbfgs}, the \texttt{slimLS} method provides good reconstructions for a much wider range of damping factors, which is a very attractive property of the \texttt{slimLS} method.
 \begin{figure}
   \begin{center}
     \includegraphics[width=.9\textwidth]{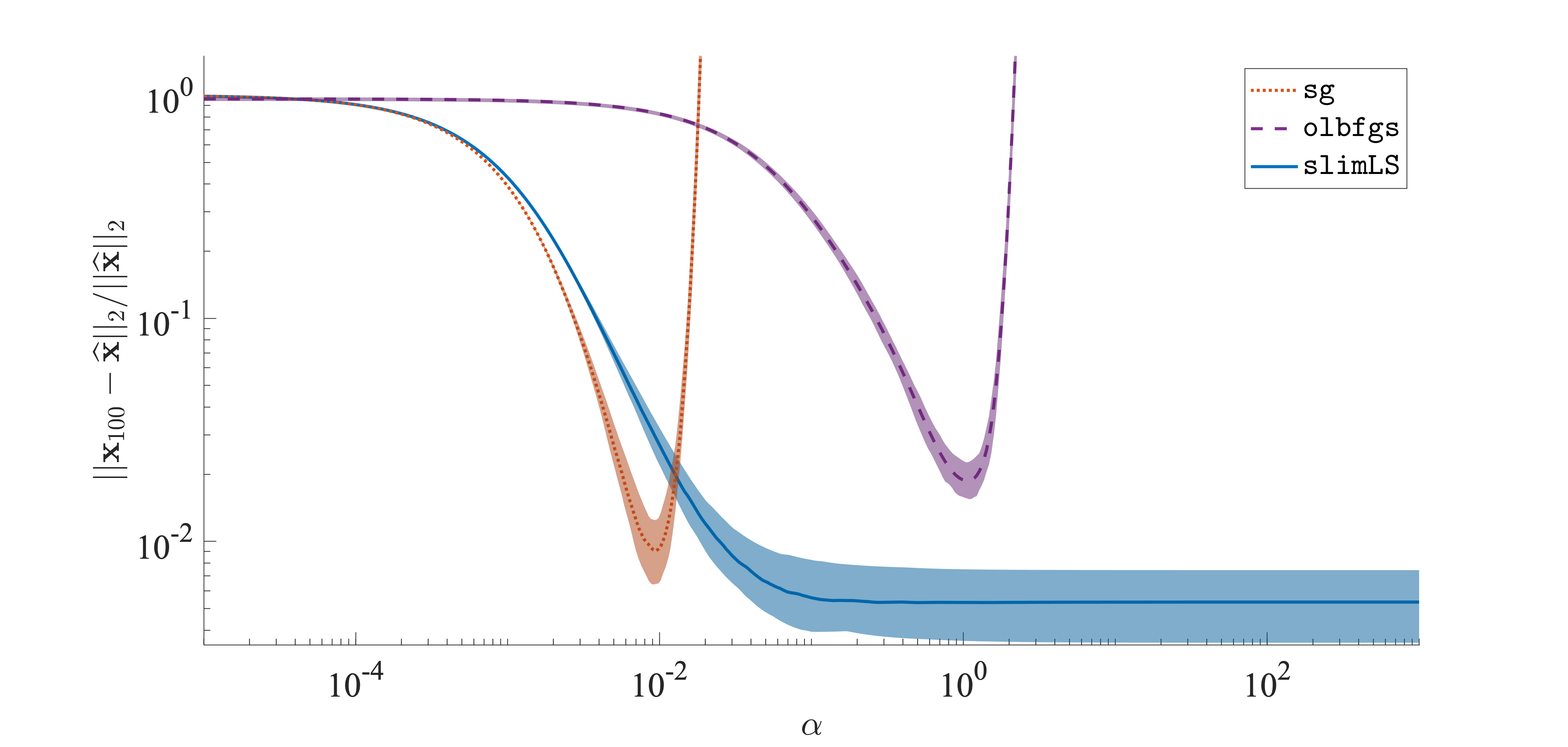}
     \caption{Comparison of median reconstruction errors along with the $5-95$th percentiles over $100$ runs and relative to $\widehat{\bfx}$ for different (fixed) step sizes in \texttt{slimLS}, sampled gradient, and online limited memory BFGS. All results correspond to accessing one epoch of the data.}\label{fig:compareStepSizing}
   \end{center}
 \end{figure}

 \subsection{Computational considerations}\label{sec:computation}
 Recall that the \texttt{slimLS} method can be interpreted as a row-action method, which by construction alleviates many of the computational bottlenecks (i.e., data access and memory requirements per iteration are significantly reduced).  However, for many realistic problems where $n$ is on the order of millions or billions (e.g., in tomography), the computational cost of each iteration can still be large.
 We remark that a noteworthy distinction of our numerical investigations compared to previously published works on row-action methods, such as \cite{Needell2014,andersen2014generalized, strohmer2009randomized}, is that we consider very large imaging problems with hundreds of millions of unknowns and focus on the initial convergence (one epoch) rather than the ``asymptotic'' convergence (hundreds of epochs) behavior.
 Next we address some considerations with respect to computational cost and comparisons with other methods.

 The \texttt{slimLS} iterates can be written as $\bfx_k = \bfx_{k-1} - \bfs_k$, where
 $$\bfs_k = \left(\alpha^{-1}_{k}\bfC_{k} + \bfM\t_{k}\bfM_{k}\right)^{-1} \bfA\t_{k}\left(\bfA_{k}\bfx_{k-1}-\bfb_{k}\right),$$
 where $\bfM_{k} = [\bfA_{k-r}\t, \ldots,  \bfA_{k}\t ]\t.$
 Thus, each iteration consists of two main steps,
 \begin{enumerate}
   \item accessing model block $\bfA_k$ and corresponding data $\bfb_k$, and
   \item computing the update step $\bfs_k$.
 \end{enumerate}

 The computational costs for the first step are often overlooked, but since data are usually stored on a hard-drive, data access can be time-consuming.  Furthermore, depending on the application, constructing the corresponding matrix block $\bfA_k$ can also be computationally expensive. In data-streaming problems one may not have control over when and which blocks of data become available at any given time.

 For the second step, solving for $\bfs_k$ can be done efficiently by first noticing that $\bfs_k$ is the solution to the LS problem,
 \begin{align} \label{eq:slimLSstep}
 \bfs_{k} = \argmin_\bfs \norm[2]{ \begin{bmatrix}
   \bfA_{k-r} \\ \vdots \\ \bfA_{k-1} \\ \bfA_k\\ \tfrac{1}{\sqrt{\alpha_k}}\bfL_k
 \end{bmatrix}\bfs - \begin{bmatrix}
   \bfzero \\ \vdots \\ \bfzero\\  \bfA_{k}\bfx_{k-1}-\bfb_{k} \\ \bfzero
 \end{bmatrix}}^2,
 \end{align}
 where $\bfC_k = \bfL_k\t \bfL_k$.
 Hence, any efficient LS solver that exploits the structure in \eqref{eq:slimLSstep} can be used to compute $\bfs_k$.
 Here, we utilize LSQR for damped LS, see \cite{PaSa82b}, where a very efficient implementation is available if $\bfL_k=\bfI$.  It is worth mentioning that another LS reformulation can be made to solve for $\bfx_k$ directly, where the right-hand-side becomes dense and will depend on $\alpha_k$ and $\bfL_k$.

 Next, we remark further on the choice of $\alpha_k$. For the \texttt{sg} method---as illustrated in Figure~\ref{fig:compareStepSizing}---an acceptable constant step size $\alpha_k$ is hard to come by and is often chosen such that $\alpha_k \ll 1$. For \texttt{olbfgs} we choose the ``natural'' constant step size $\alpha_k = 1$, with a possible exception at the early iterations. Since \texttt{olbfgs} is equivalent to {\tt sg} in the first iteration, \texttt{olbfgs} may suffer from large step sizes while building up its memory. To compensate for this we can ramp up the value of $\alpha_k$ in early iterations, e.g., $\alpha_k =  \frac{k \alpha}{r+1}$ for the first $r+1$ iterations for fixed $\alpha$ where $r$ is the memory parameter.

 In many structured problems and in particular for tomography problems---where each block matrix corresponds to a single projection image (i.e., one angle)---$\bfA_k$ is extremely sparse with the number of non-zero elements in $\bfA_k$ on the order of $n$. Note that in some cases, matrix $\bfA_k$ does not even need to be constructed, but function evaluations can be used within iterative methods \cite{golub2012matrix}. Hence, for extremely large-scale problems where memory storage becomes an issue, {\tt slimLS} can take advantage of any sparsity or structure. In contrast, the \texttt{olbfgs} method requires storing two vectors of length $n$ for each memory level, and these vectors are likely dense so the storage becomes cumbersome for large $n$.

 \subsection{Large-scale Tomographic Reconstruction}\label{sec:num2}
 Next we present two numerical experiments that demonstrate the performance of \texttt{slimLS} for solving massive tomography reconstruction problems.
 Tomography has become very important in many applications, including medical imaging, seismic imaging, and atmospheric imaging \cite{natterer2001mathematics,herman1993algebraic}. The goal in computerized tomography is to reconstruct the interior of an object given observed, exterior measurements.  However, recent advances in detector technology have led to faster scan speeds and the collection of massive amounts of data.  Furthermore, in dynamic or streaming data scenarios (e.g., in microCT reconstruction), partial reconstructions are needed to inform the data acquisition process.  The $\texttt{slimLS}$ method can be used to address both of these problems.  The first example we consider is a very large, limited angle 2D tomography reconstruction problem that is underdetermined, and the second example is a 3D streaming reconstruction problem that is ultimately overdetermined.

 For ill-posed problems such as tomography, semiconvergence of iterative methods is a concern whereby early iterates converge quickly to a good approximation of the solution but later iterates become contaminated with errors.  Iterative regularization techniques (i.e., early termination of the iterative process) are often used to obtain a reasonable solution \cite{hansen2010discrete}.
 For row-action methods applied to tomography problems, semiconvergence properties have been investigated \cite{elfving2014semi}, but due to notoriously slow convergence (after fast initial convergence), the ill effects of semiconvergence tend to appear only after multiple epochs of the data.
 Thus, we do not include additional regularization in the following results.

 \begin{figure}[b]
  \begin{center}
 \begin{tabular}{cc}
   \includegraphics[height=.5\textwidth]{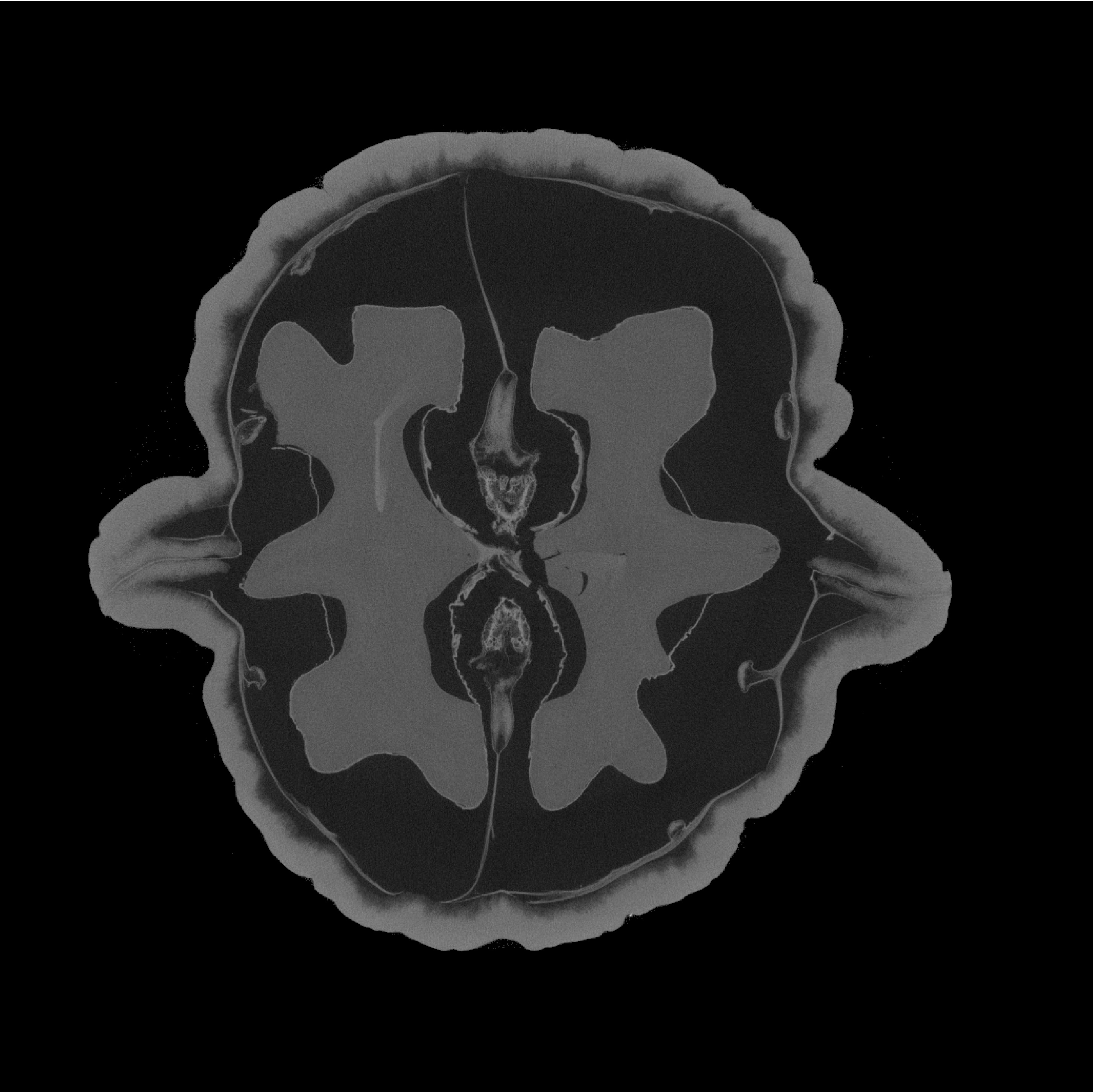}&
   \includegraphics[height=.5\textwidth]{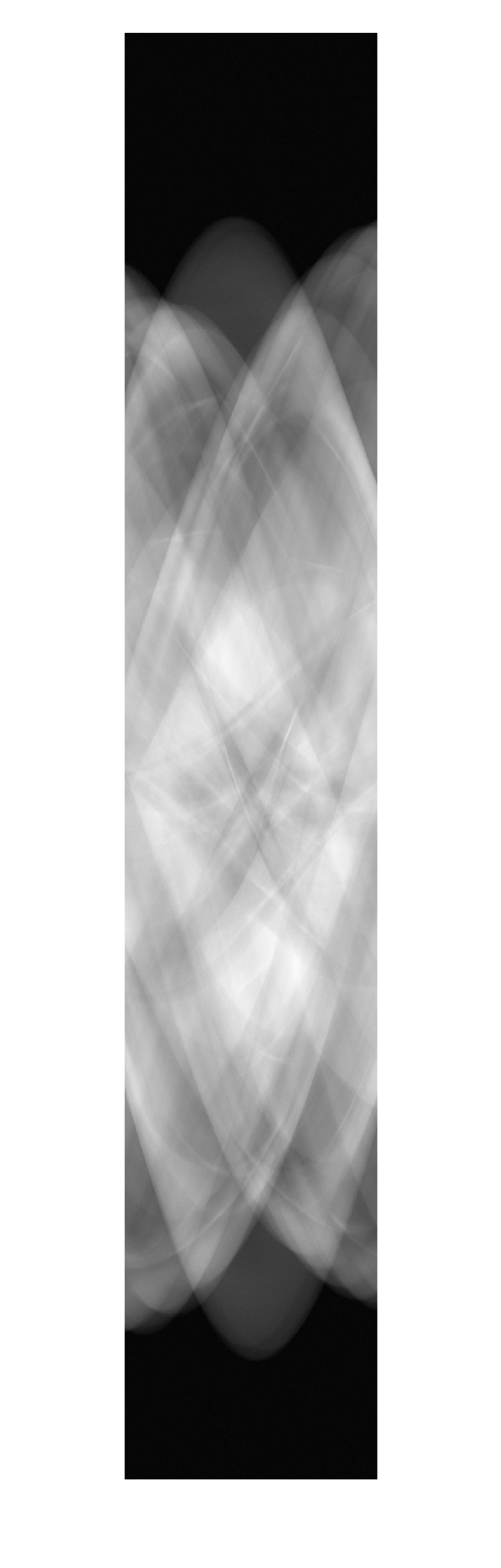}\\
   (a) & (b)
 \end{tabular}
   \end{center}
   \caption{Two-dimensional tomography problem with missing wedge of $60^\circ$.  True image of a walnut slice is provided in (a) and the observed sinogram in (b) corresponds to angles between $-60^\circ$ and $+60^\circ$ at $0.3$ degree steps.}\label{fig:tomo2d}
 \end{figure}

 \paragraph{Two-Dimensional Limited-Angle Tomography.}
 First we consider a parallel-beam x-ray tomography example, where the true image represents a cross-section of a walnut.  In \cite{hamalainen2015tomographic} the authors provide an image reconstruction computed from $1200$ projections using filtered back projection.  Our ``ground truth" image provided in Figure~\ref{fig:tomo2d}~(a) is a cleaned image of the filtered back projection reconstruction.  For this problem, $\bfx_\true\in \bbR^{2,296^2}$, and we simulate observations by taking $400$ projections at angles between $-60^\circ$ and $+60^\circ$ at $0.3$ degree steps, with $2,\!296$ rays for each angle.  This can be interpreted as a missing data problem where we have a missing wedge of $60^\circ$. In this example, $\bfA \in \bbR^{400\cdot2,296 \times 2,296^2}$.  The observed sinogram provided in Figure~\ref{fig:tomo2d}~(b) was generated as in \eqref{eq:inverseproblem} with noise level $0.01$.

Here $\bfW\t\bfA$ is random cyclic uniformly chosen from $400$ blocks of size $2,\!296 \times 2,\!296^2$.  We apply the \texttt{slimLS} method with memory level $r = 2$ and ramped-up damping parameter with $\alpha = 1$.
Relative reconstruction errors computed as $\norm[2]{\bfx_k - \bfx_\true}/\norm[2]{\bfx_\true}$ are provided in Figure~\ref{fig:tomo2d_err}. We provide comparisons to the \texttt{sg} method with step size $\alpha_k = 10^{-5}$ and the \texttt{olbfgs} method with memory level $20$ and ramped-up step size with $\alpha = 1$.  We found that both methods required a small initial step size to prevent reconstruction errors from getting very large.
 \begin{figure}
 \begin{center}
   \includegraphics[width=.8\textwidth]{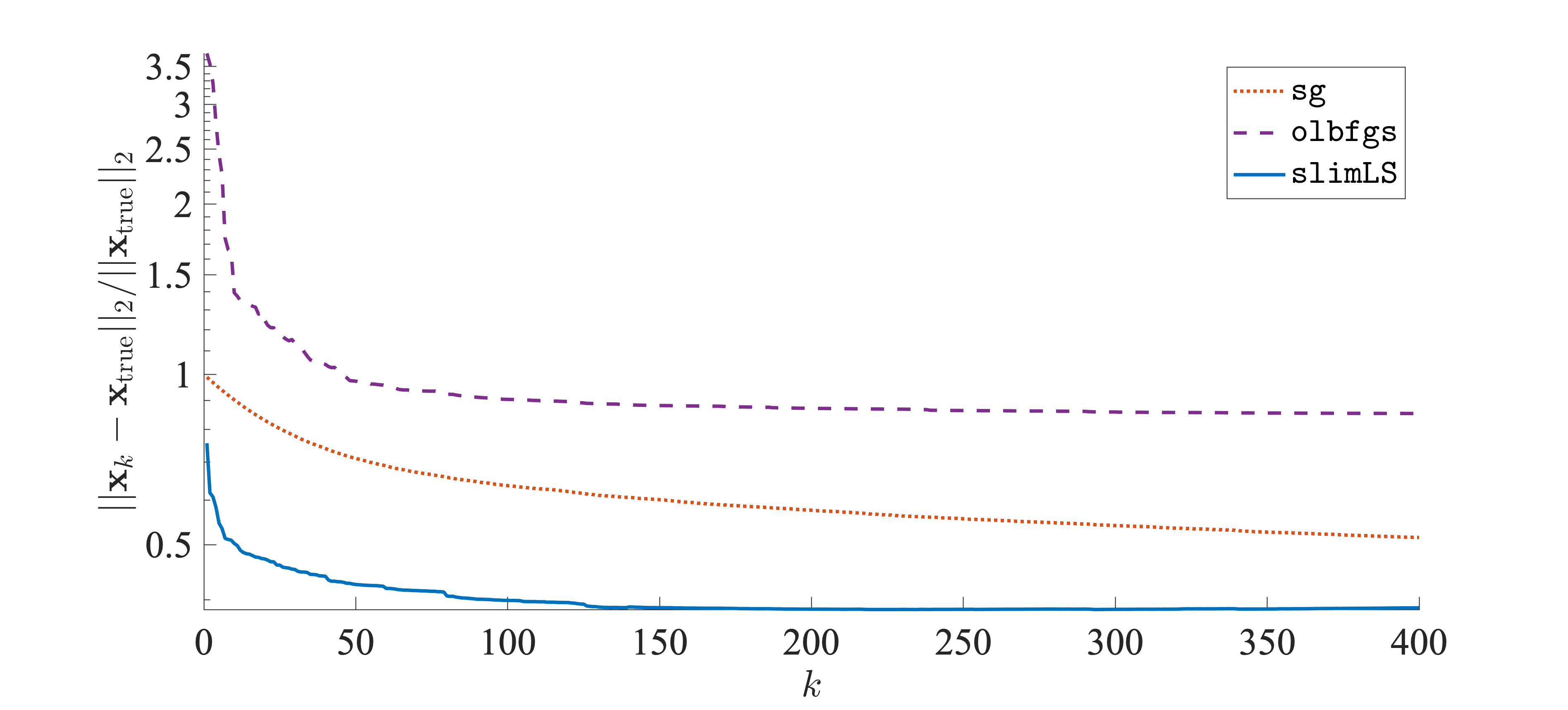}
 \caption{Relative reconstruction error norms at each iteration of \texttt{sg}, \texttt{olbfgs}, and \texttt{slimLS}.}\label{fig:tomo2d_err}
 \end{center}
 \end{figure}

 \begin{figure}[bth]
 \begin{center}
   \begin{tabular}{cccc}
       $\bfx_{\rm true}$ & {\tt sg} & {\tt olbfgs} & {\tt slimLS} \\
       \includegraphics[width=0.22\textwidth]{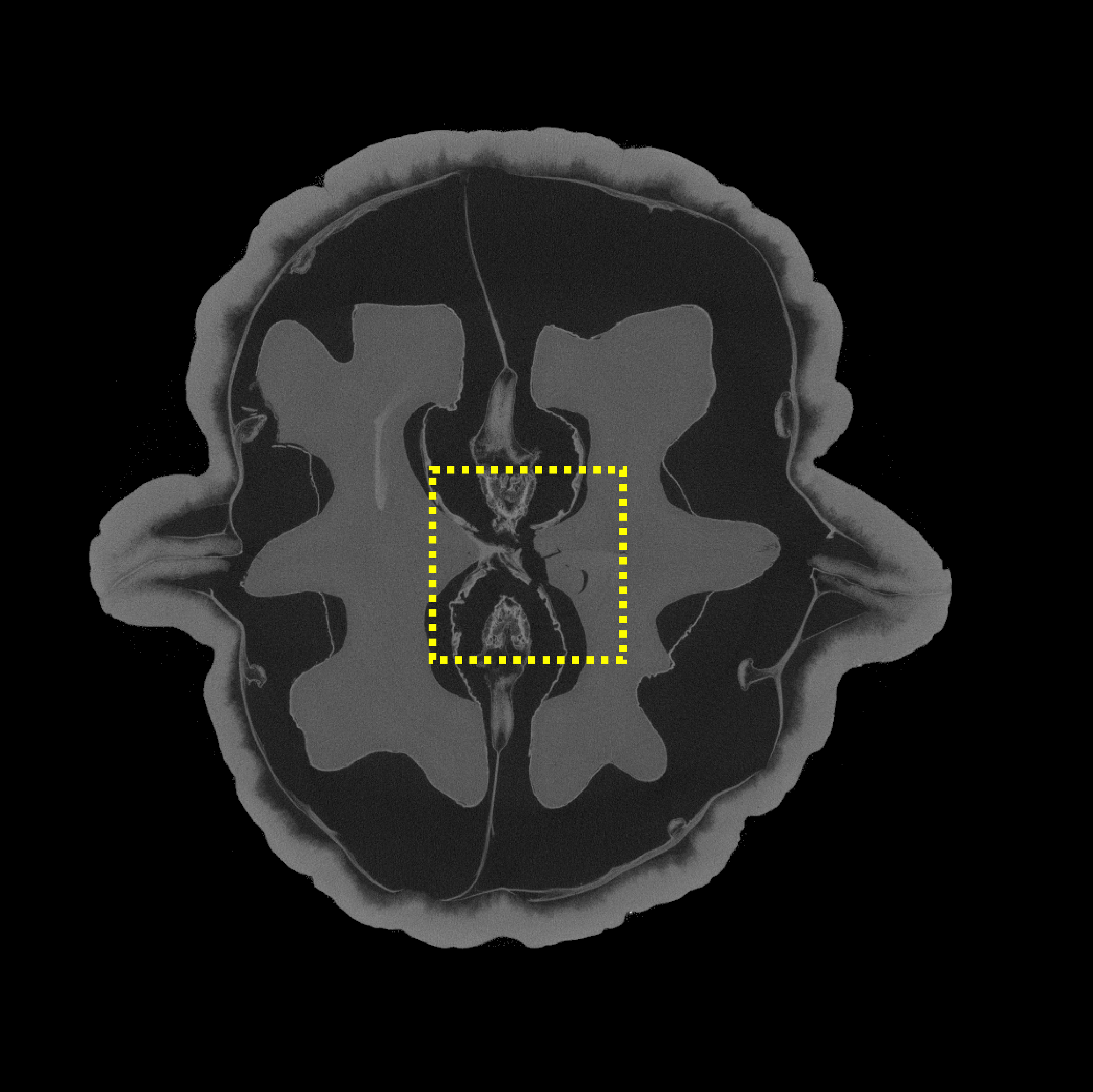}&
       \includegraphics[width=0.22\textwidth]{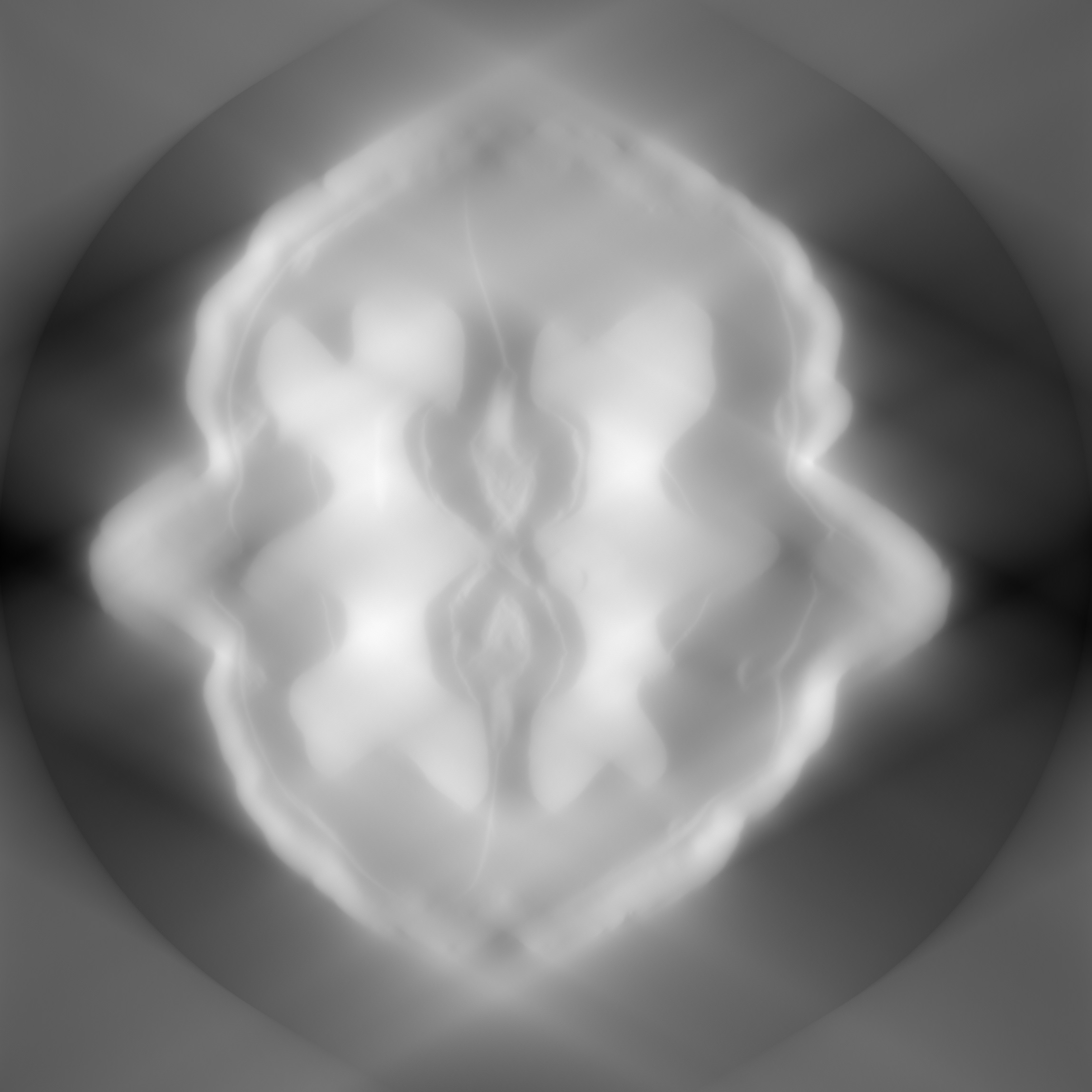} &
       \includegraphics[width=0.22\textwidth]{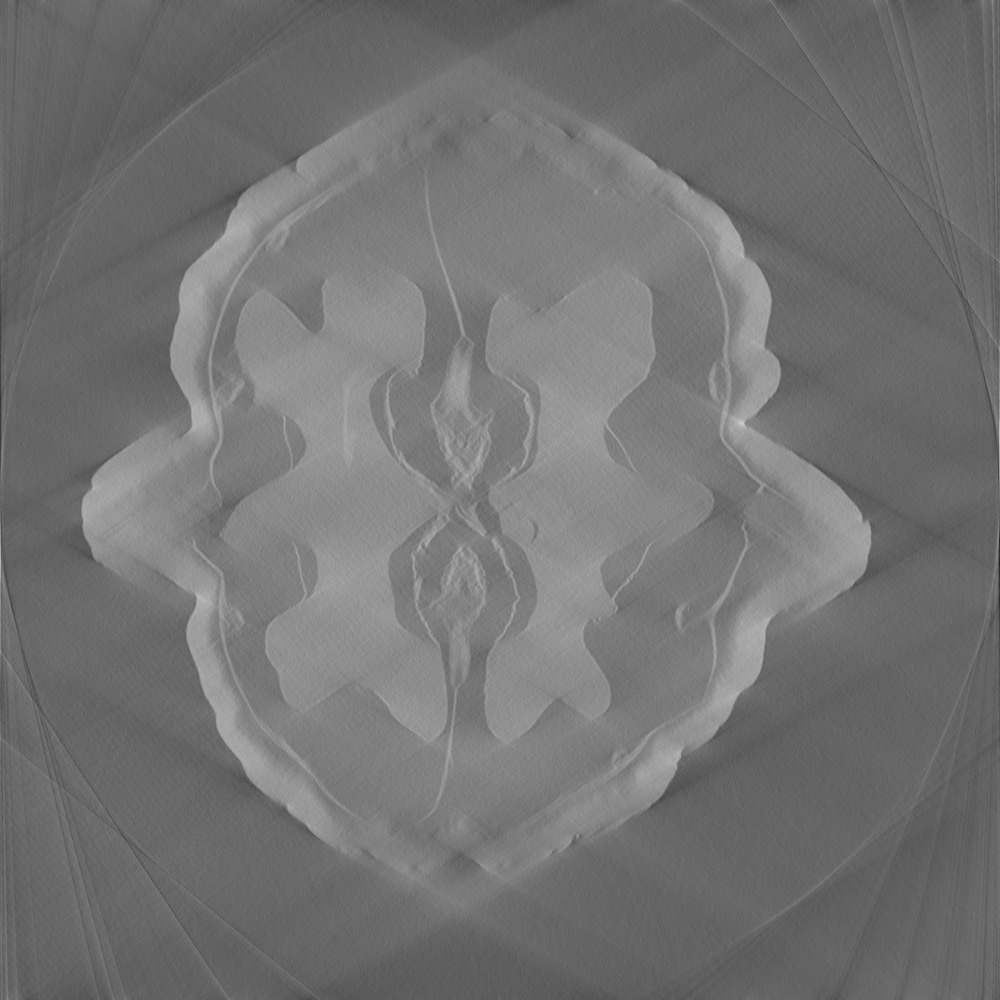}&
       \includegraphics[width=0.22\textwidth]{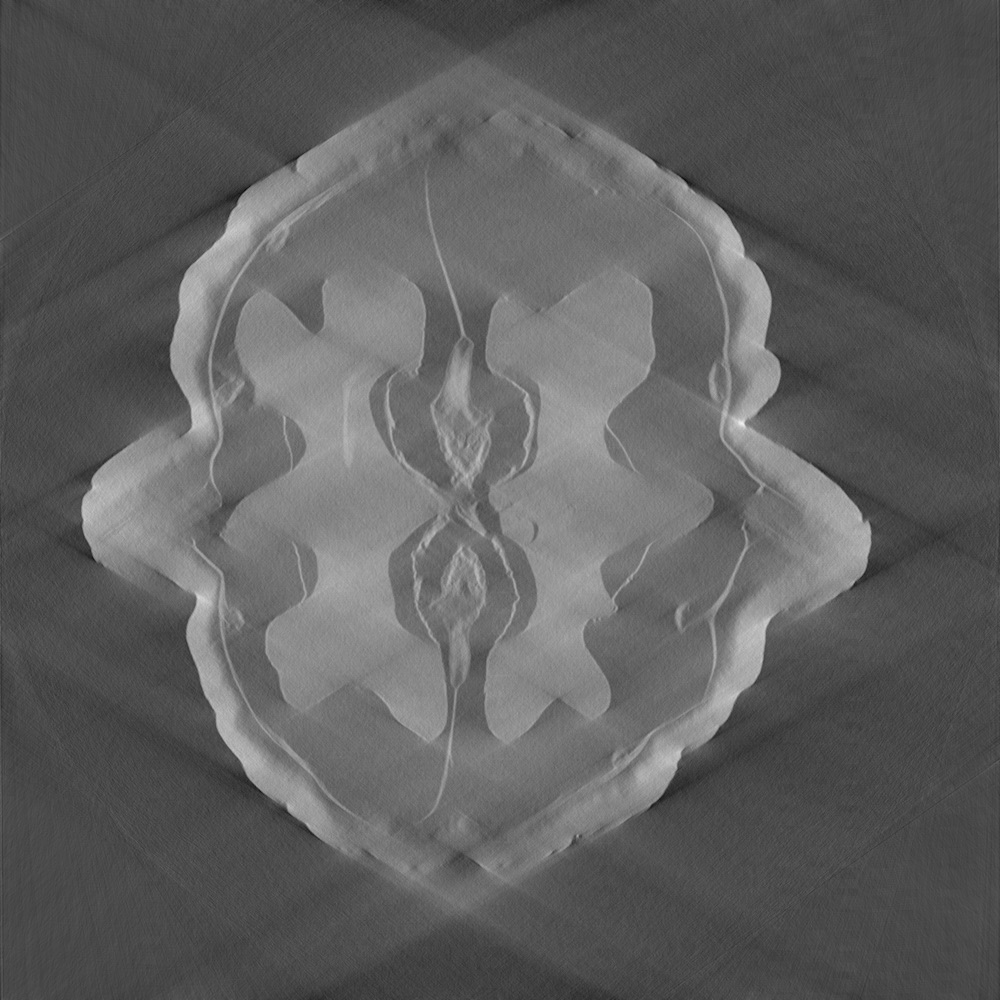}\\
       \includegraphics[width=0.22\textwidth]{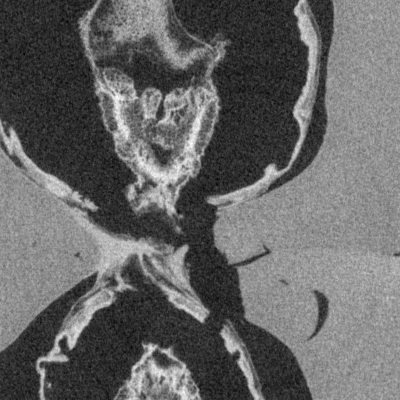}&
       \includegraphics[width=0.22\textwidth]{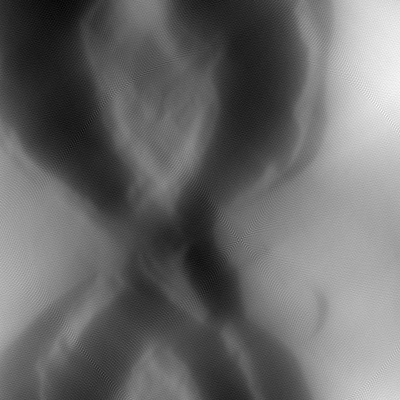} &
       \includegraphics[width=0.22\textwidth]{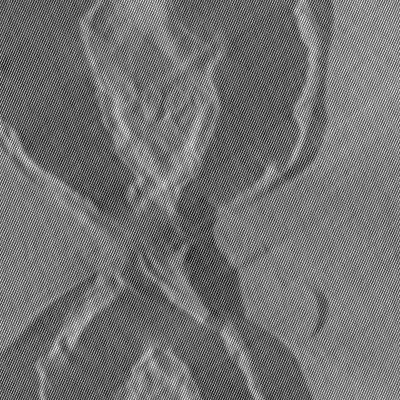}&
       \includegraphics[width=0.22\textwidth]{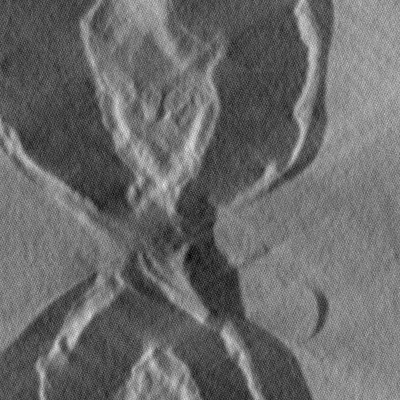}
   \end{tabular}
 \caption{Image reconstructions for \texttt{sg}, \texttt{olbfgs} and \texttt{slimLS}, along with true image and with subimages provided below.  The area of the subimage is denoted by the enclosed dotted region on the true image.}\label{fig:tomo2d_images}
 \end{center}
 \end{figure}

 Image reconstructions and subimages, along with the true image, are provided in Figure~\ref{fig:tomo2d_images}.  We observe that after one epoch of the data, the \texttt{slimLS} reconstruction contains sharper details and fewer artifacts than the \texttt{sg} and \texttt{olbfgs} reconstructions.  As described in Section~\cref{sec:computation}, it is difficult to provide a fair comparison of methods, especially in terms of the memory level and the step length.  Careful tuning of the step length for \texttt{sg} and \texttt{obfgs} can lead to reconstructions that are similar in quality to the \texttt{slimLS} reconstruction, but that is very time consuming especially for these massive problems.  Also, as observed in Section~\cref{sec:num1}, there may be only a small window of good values.
 If good parameters are known in advance, \texttt{olbfgs} can produce good solutions, but if they are not known a priori, then poor results, or even divergence of the relative reconstruction errors, were often observed.

 \begin{figure}[b]
   \begin{center}
 \begin{tabular}{ccc}
   \raisebox{-.5\height}{\includegraphics[width=.5\textwidth]{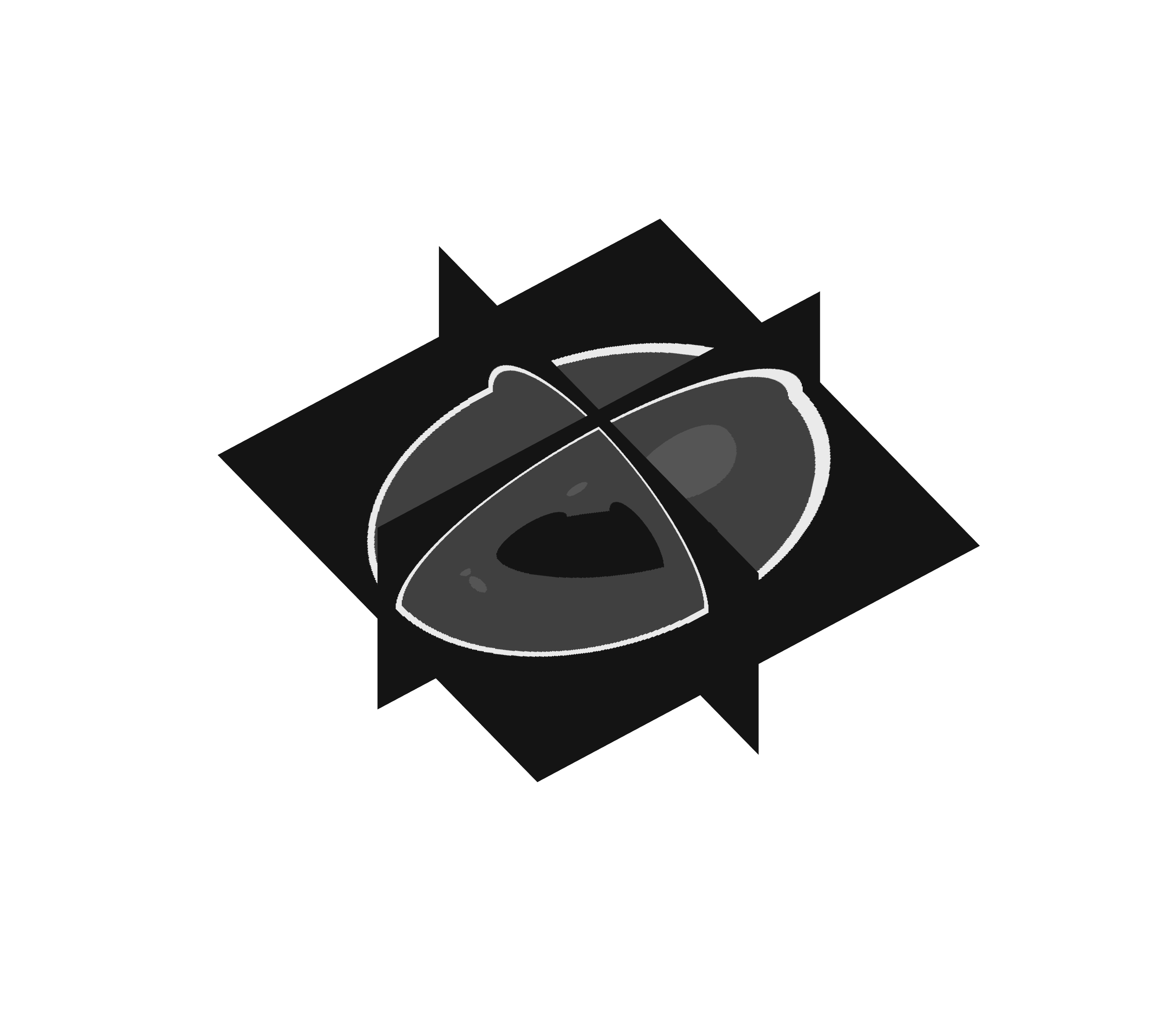}}&
   \includegraphics[width=.2\textwidth]{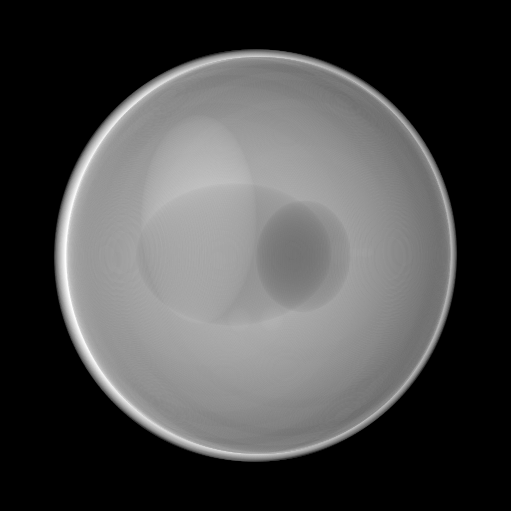}
   & \includegraphics[width=.2\textwidth]{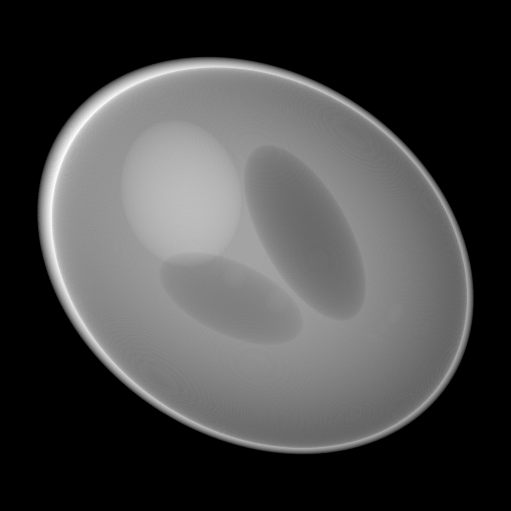}\\[-18ex]
   & \includegraphics[width=.2\textwidth]{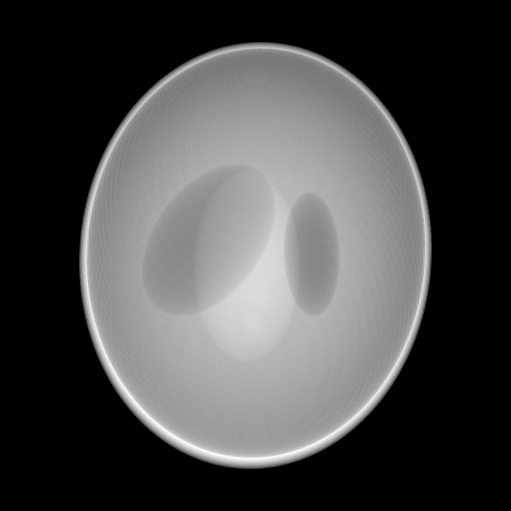}&
   \includegraphics[width=.2\textwidth]{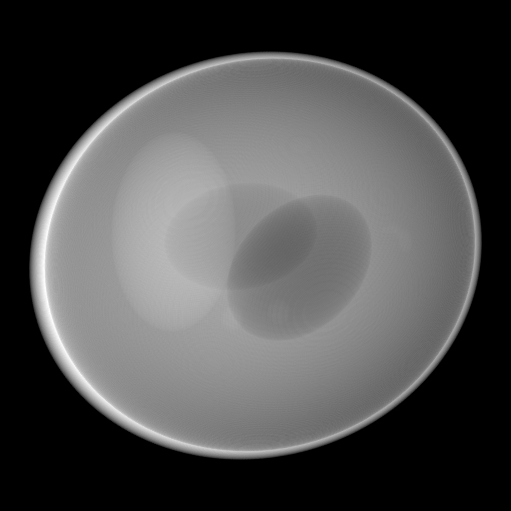}
 \end{tabular}
   \caption{For the 3D tomography example, we provide orthogonal slices of the true image phantom (left) along with four of the observed projection images (right).}\label{fig:tomo3d}
   \end{center}
 \end{figure}

 \paragraph{Three-Dimensional Streaming Tomography.}
 Next we demonstrate the power of \texttt{slimLS} for three-dimensional tomography reconstruction. For very large problems where data access is a computational bottleneck and for problems where data is being streamed and partial reconstructions are needed, row-action methods are the only feasible option.   Here we show that sampled limited memory methods can be a good alternative.

 In this problem setup, the true image is a $511 \times 511 \times 511$ modified 3D Shepp-Logan phantom.  We generate $1,\!000$ projection images of size $511 \times 511$ taken from random directions, which are samples from the uniform distribution over the unit sphere. The raytracing matrix $\bfA$ is of size $261,\!121,\!000 \times 133,\!432,\!831$ and is certainly never constructed, but matrix blocks are generated using a modification of the \texttt{tomobox} code \cite{tomobox} which represents parallel beam tomography. White noise was added to the projection images such that the noise level over all projection images is $0.001$.  Orthogonal slices of the true 3D phantom, along with four of the observed, projection images are provided in Figure~\ref{fig:tomo3d}.

 We run the \texttt{slimLS} method with $r=0$ and damping factor $\alpha_k = 1$.
 After one epoch of the data (i.e., the cost of accessing all data once), we
 compare the \texttt{slimLS} reconstruction to the stochastic gradient and online LBFGS reconstructions. For \texttt{sg}, we use step length $\alpha_k = 0.0001$, and for \texttt{olbfgs}, we use $\alpha_k = 1$ and memory level $10$.  Relative reconstruction error norms per iteration are provided in Figure~\ref{fig:tomo3dRelError}.  Notice that even in early iterations, where only a small fraction of the data has been accessed, \texttt{slimLS} reconstructions have smaller relative reconstruction errors than \texttt{sg} and \texttt{olbfgs}.

 We computed absolute error images, which correspond to absolute values of the difference between the reconstruction and the true image, and we provide three slices (in the $x$, $y$, and $z$ direction) for \texttt{sg}, \texttt{olbfgs}, and \texttt{slimLS} in Figure \ref{fig:errimages}.  The absolute errors are provided in inverted colormap so that black corresponds to large errors.  All images use the same colormap. We observe that absolute error images for \texttt{sg} are significantly worse than those from \texttt{olbfgs} and \texttt{slimLS} reconstructions, with absolute error images for \texttt{slimLS} having fewer artifacts.

 \begin{figure}[b]
  \begin{center}
   \includegraphics[width=1.0\textwidth]{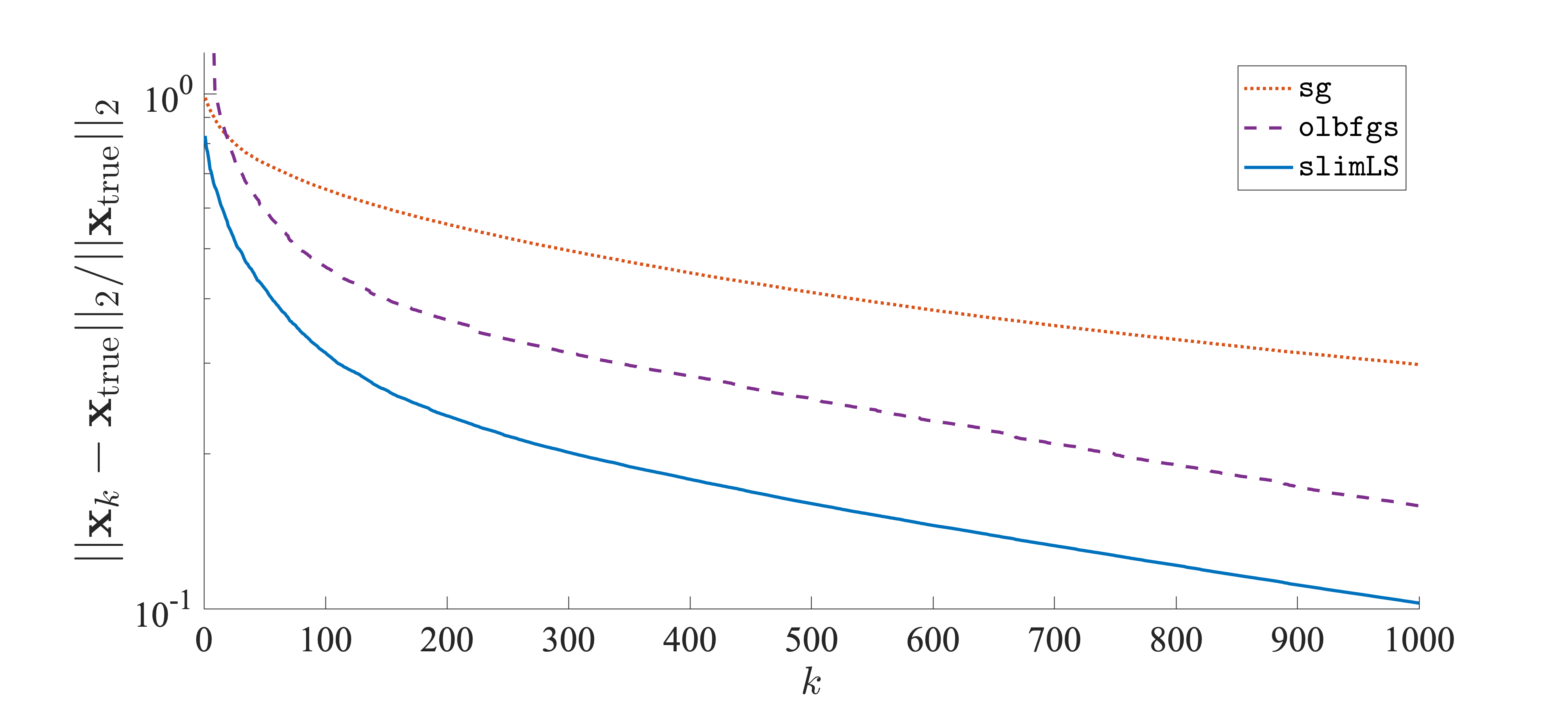}
   \end{center}
   \caption{Relative reconstruction errors per iteration of \texttt{sg}, \texttt{olbfgs}, and \texttt{slimLS} for the 3D tomography example.}\label{fig:tomo3dRelError}
 \end{figure}

 \begin{figure}[bthp]
  \begin{center}
 \begin{tabular}{cccc}
   & {\tt sg}   & {\tt olbfgs} & {\tt slimLS} \\
   \rotatebox{90}{\hspace*{7ex}$x$-slice}&
   \includegraphics[width=.3\textwidth]{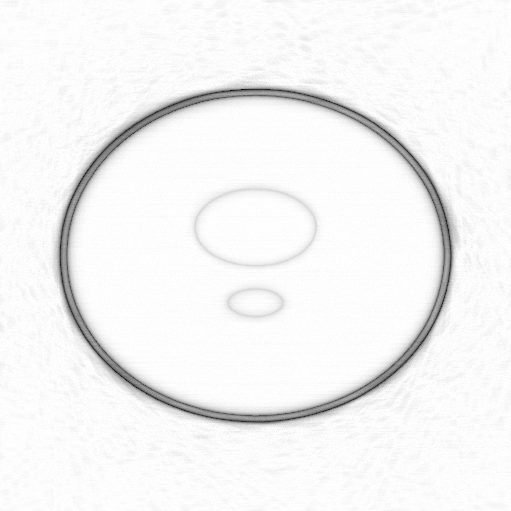}&
   \includegraphics[width=.3\textwidth]{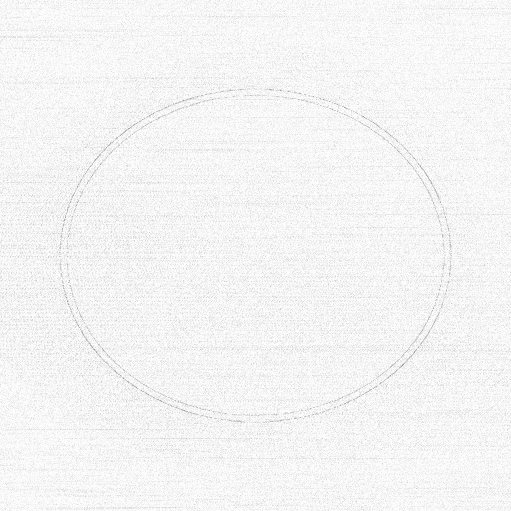}&
   \includegraphics[width=.3\textwidth]{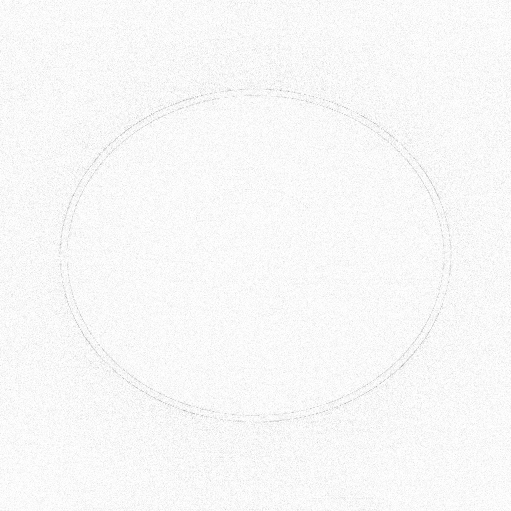}\\[-4ex]
   & \multicolumn{1}{r}{$0.3262$} & \multicolumn{1}{r}{$0.1548$} & \multicolumn{1}{r}{$0.1005$}\\ [2ex]
   \rotatebox{90}{\hspace*{7ex}$y$-slice} &
   \includegraphics[width=.3\textwidth]{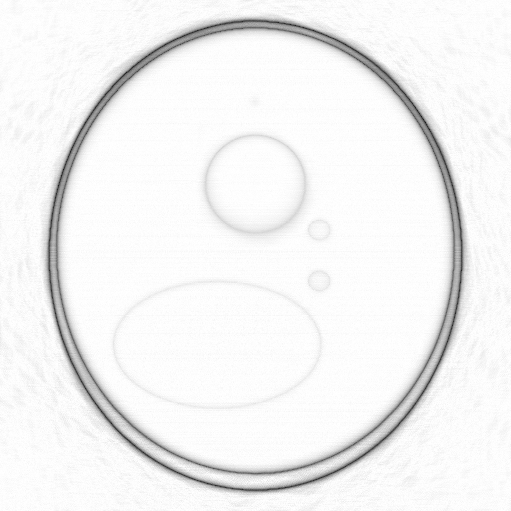}&
   \includegraphics[width=.3\textwidth]{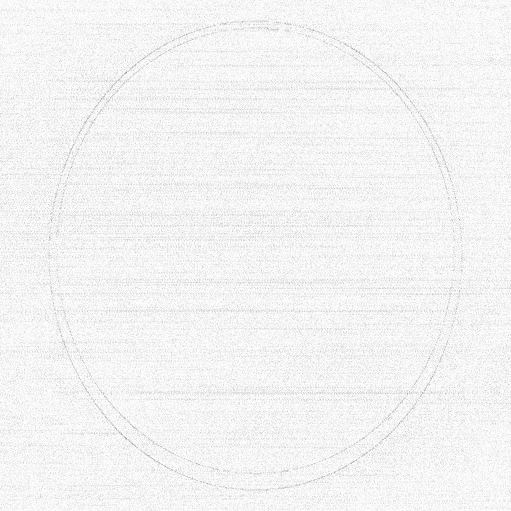}&
   \includegraphics[width=.3\textwidth]{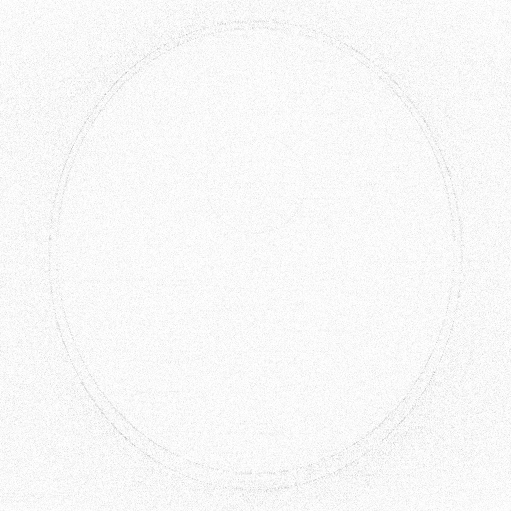}\\[-4ex]
   & \multicolumn{1}{r}{$0.2580$} & \multicolumn{1}{r}{$0.1248$} & \multicolumn{1}{r}{$0.0750$}\\ [2ex]
   \rotatebox{90}{\hspace*{7ex}$z$-slice} &
   \includegraphics[width=.3\textwidth]{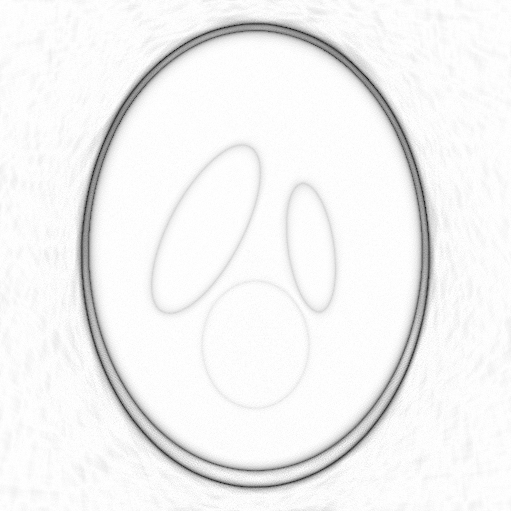}&
   \includegraphics[width=.3\textwidth]{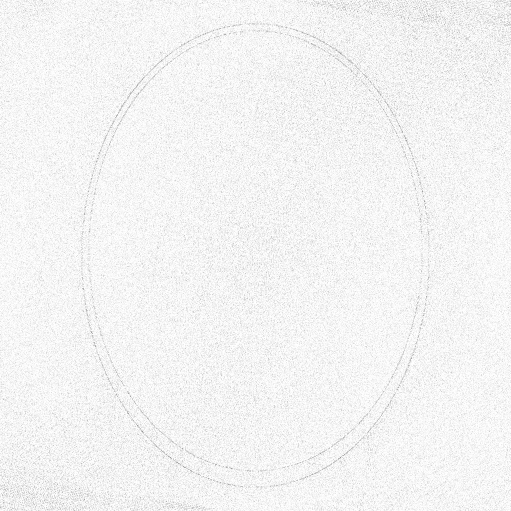}&
   \includegraphics[width=.3\textwidth]{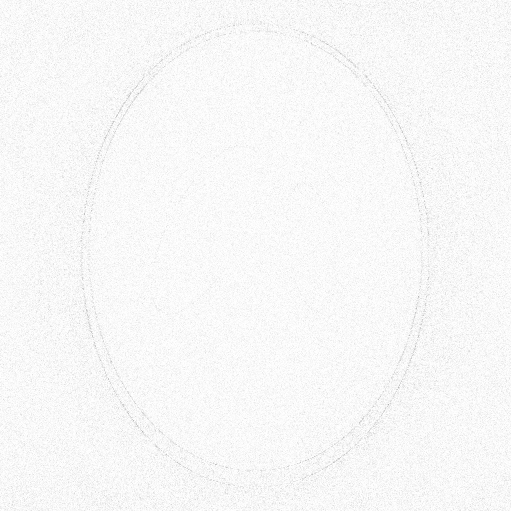}\\[-4ex]
   & \multicolumn{1}{r}{$0.2738$} & \multicolumn{1}{r}{$0.1386$} & \multicolumn{1}{r}{$0.0873$}  \\
 \end{tabular}
   \caption{For the 3D tomographic reconstruction problem, we provide $x$, $y$, and $z$ slices of absolute error images for \texttt{sg}, \texttt{olbfgs}, and \texttt{slimLS} reconstructions after one epoch of the data.  The relative reconstruction error norm for each slice is provided in the bottom right corner of each image.}\label{fig:errimages}
   \end{center}
 \end{figure}

 \section{Conclusions}\label{sec:con}
 In this paper, we investigate sampled limited memory methods for solving massive inverse problems, such as those that arise in modern 2D and 3D tomography applications.  Limited memory row-action methods are relevant in scenarios where \emph{full} matrix-vector-multiplications with coefficient matrix $\bfA$ are not possible or too computationally expensive.  This includes problems where $\bfA$ is so large that it does not fit in computer memory and problems where the data is being streamed.
 The main theoretical contribution is that, contrary to existing row-action and stochastic approximation methods, the \texttt{slimLS} method has both favorable initial and asymptotic convergence properties. Additional benefits of \texttt{slimLS} include faster initial convergence due to the use of information from previous iterates and convergence for a wider range of step length or damping parameters.  We provide theoretical convergence results, and numerical examples from massive tomography reconstruction problems show the potential impact of these methods.

\section*{Acknowledgements}
We gratefully acknowledge support by the National Science Foundation under grants NSF DMS 1723048 (L. Tenorio), NSF DMS 1723005 (M. Chung, J. Chung) and NSF DMS 1654175 (J. Chung).  J. Chung and M. Chung would also like to acknowledge the Alexander von Humboldt Foundation for their generous support.

\appendix

\section{Randomized recursive LS method} \label{sec:proofs2}
In this section, we show that the $k$-th iterate of the randomized recursive LS method defined in \eqref{eq:rrls} is the minimum norm solution of LS problem in \eqref{eq:rrls2}.	We prove this by induction. For $k=1$ and with $\bfx_0=\bfzero$,  \eqref{eq:rrls} yields  $\bfx_{1} = \bfA^{\dagger}_{1}\bfb_{1}$.  Now let us assume that
\begin{align} \bfx_{k-1} = \left(\sum_{i=1}^{k-1} \bfA\t_{i}\bfA_{i}\right)^{\dagger} \begin{bmatrix} \bfA_{1} \\ \vdots \\ \bfA_{k-1} \end{bmatrix}\t  \begin{bmatrix} \bfb_{1} \\ \vdots \\ \bfb_{k-1} \end{bmatrix}, \end{align}
then from \eqref{eq:rrls}, we get
\begin{align} \nonumber
   \bfx_{k} &  = \bfx_{k-1} - \left(\sum_{i=1}^{k}\bfA\t_{i}\bfA_{i}\right)^{\dagger}\bfA\t_{k}\left(\bfA_{k}\bfx_{k-1}-\bfb_{k}\right) \\
	 & = \left(\sum_{i=1}^{k}\bfA\t_{i}\bfA_{i}\right)^{\dagger}\left[\left(\sum_{i=1}^{k}\bfA\t_{i}\bfA_{i}\right)\bfx_{k-1} - \bfA\t_{k}\left(\bfA_{k}\bfx_{k-1}-\bfb_{k}\right)\right] \label{eq:range}
\end{align}
where we are using the fact that $ \bfx_{k-1} = \left(\sum_{i=1}^{k}\bfA\t_{i}\bfA_{i}\right)^{\dagger}\left(\sum_{i=1}^{k}\bfA\t_{i}\bfA_{i}\right)\bfx_{k-1}$ since
$\bfx_{k-1} \in  \calR \left(\sum_{i=1}^{k-1}\bfA\t_{i}
\bfA_{i}\right) \subseteq  \calR \left(\sum_{i=1}^{k}\bfA\t_{i}\bfA_{i}\right)$.
A similar argument can be used to show that $ \left(\sum_{i=1}^{k-1} \bfA\t_{i}\bfA_{i}\right)\left(\sum_{i=1}^{k-1} \bfA\t_{i}\bfA_{i}\right)^{\dagger} \begin{bmatrix} \bfA_{1} \\ \vdots \\ \bfA_{k-1} \end{bmatrix}\t  \begin{bmatrix} \bfb_{1} \\ \vdots \\ \bfb_{k-1} \end{bmatrix}  = \begin{bmatrix} \bfA_{1} \\ \vdots \\ \bfA_{k-1} \end{bmatrix}\t  \begin{bmatrix} \bfb_{1} \\ \vdots \\ \bfb_{k-1} \end{bmatrix} $ since $\begin{bmatrix} \bfA_{1} \\ \vdots \\ \bfA_{k-1} \end{bmatrix}\t  \begin{bmatrix} \bfb_{1} \\ \vdots \\ \bfb_{k-1} \end{bmatrix} \in \calR \left(\sum_{i=1}^{k-1} \bfA\t_{i}\bfA_{i}\right)$, and we arrive at
\begin{equation*}
   \bfx_{k} =   \left(\sum_{i=1}^{k}\bfA\t_{i}\bfA_{i}\right)^{\dagger}\left[\left(\sum_{i=1}^{k-1}\bfA\t_{i}\bfA_{i}\right)\bfx_{k-1} + \bfA\t_{k}\bfb_{k}\right]
     = \left(\sum_{i=1}^{k}\bfA\t_{i}\bfA_{i}\right)^{\dagger} \begin{bmatrix} \bfA_{1} \\ \vdots \\ \bfA_{k} \end{bmatrix}\t  \begin{bmatrix} \bfb_{1} \\ \vdots \\ \bfb_{k} \end{bmatrix}.
\end{equation*}

\section{Proofs for Section \ref{sec:conv}} \label{sec:proofs}

We continue to use the notation and assumptions from Sections \ref{sec:intro} and \ref{sec:conv}.

\begin{lemma} \label{lem:B}
 If $\bfA \in \mathbb{R}^{m \times n}$ has full column-rank, then for any fixed $\alpha>0$:
 \begin{enumerate}
\vspace{-.1cm}
\item $\norm[2]{\mathbb{E} \bfB_k/\alpha} < 1$ with upper bound
\[
\|\mathbb{E}\,\bfB_k/\alpha\|_2 \leq \frac{M^o}{M} + \frac{M-M^o}{M(1+\alpha\,A_{\min})} < 1,
\]
where $M^o$ is the number of eigenvalues $\lambda_{\min}(\bfA\t\bfW^{(i)}\left(\bfW^{(i)}\right)\t\bfA)$ equal to zero over the
$M$ different $\bfW^{(i)}$, \label{lem:partone}
\item $\bfB$ is symmetric positive definite, \label{lem:parttwo}
\item $\norm[2]{\bfB_{k}\bfA\t_{k}\bfA_{k}} \leq {\alpha A_{\max}}/{(1+\alpha A_{\max})}$, and \label{lem:partfour}
\item $0 < {\alpha A_{\min}}/{[M (1+\alpha A_{\min})]} \leq \lambda_{\min}(\bfB) \leq \lambda_{\max}(\bfB) < {(1+\alpha A_{\max})}/{2}$.     \label{lem:partthree}
\end{enumerate}
 \end{lemma}
 \begin{proof}(i) For simplicity we use the notation $\widetilde{\bfA}_i = (\bfW^{(i)})^\top\bfA$. By Jensen's inequality
\[
\|\mathbb{E}\,\bfB_k/\alpha \|_2= \|\mathbb{E}\,(\bfI+\alpha\,\bfA_k^\top\bfA_k)^{-1}\|_2 \leq
\mathbb{E}\,\| (\bfI+\alpha\,\bfA_k^\top\bfA_k)^{-1}\|_2 \leq 1.
\]
The last inequality becomes an equality only if $\| (\bfI+\alpha\,\bfA_k^\top\bfA_k)^{-1}\|_2=1$ a.s., in which case
$\lambda_{\max}(\bfA_k^\top\bfA_k) =0$ across all realizations of $\bfW_k$. That is,
$\| (\bfI+\alpha\,\widetilde{\bfA}_i^\top\widetilde{\bfA}_i)^{-1}\|_2 =1$ for all $i$. Furthermore, an eigenvector
$\bfv$ for the largest eigenvalue $\lambda=1$ of $\mathbb{E}\,\| (\bfI+\alpha\,\bfA_k^\top\bfA_k)^{-1}\|_2$
is also an eigenvector of all $(\bfI+\alpha\,\widetilde{\bfA}_i^\top\widetilde{\bfA}_i)^{-1}$ with eigenvalue 1.
That is, $\widetilde{\bfA}_i^\top\widetilde{\bfA}_i\,\bfv = \mathbf{0}$ for all $i$. This is not possible
because $\mathbb{E}\,\bfA_k^\top\bfA_k = \beta\,\bfA^\top\bfA$. Hence, $\|\bfB_k/\alpha\|_2<1$. This implies that $M^o <M$ and the upper bound in (i) follows again from
$\|\mathbb{E}\,\bfB_k/\alpha\|_2 \leq \mathbb{E}\,\|(\bfI + \alpha\,\bfA_k^\top\bfA_k)^{-1}\|_2$.\\
(ii) It is clear that $\bfB$ is symmetric positive semi-definite because $\bfB=\bfI -\mathbb{E}\,\bfB_k/\alpha$. It is positive definite by (i).\\
(iii) The upper bound follows from the identity $\bfB_k\bfA_k^\top\bfA_k = \bfI - (\bfI + \alpha\,\bfA_k^\top\bfA_k)^{-1}$.\\
(iv) From (iii),
\[
\lambda_{\max}(\bfB)\leq  \frac{\alpha A_{\max}}{1+\alpha A_{\max}} < \frac{1+\alpha\,A_{\max}}{2}.
\]
The lower bound follows from (i):
\begin{eqnarray*}
\lambda_{\min}(\bfB) \geq 1 - \mathbb{E}\,\|\bfB_k/\alpha\|_2 &\geq &
1 -  \frac{M^o}{M} - \frac{M-M^o}{M(1+\alpha\,A_{\min})}\\
& = & \left(\frac{M-M^o}{M}\right)\frac{\alpha\,A_{\min}}{1 + \alpha\,A_{\min}}
\geq  \frac{\alpha\,A_{\min}}{M(1 + \alpha\,A_{\min})}.
\end{eqnarray*}
\end{proof}

\begin{proof}[\textbf{Proof of Theorem \ref{thm:firstmoment}}]
We use $(\mathcal{F}_k)$ to denote the natural filtration induced by the sequence $\{\bfW_k\}$:
$\mathcal{F}_k = \sigma(\bfW_j,\,j\leq k)$.\\
(i) Using the recursion for $\bfx_{k}$, we obtain
\begin{align*}
\mathbb{E}\left[\bfx_{k}  - \widehat{\bfx} \rvert \, \mathcal{F}_{k-1} \, \right] & =  \bfx_{k-1} - \widehat{\bfx} - \mathbb{E}\left[\bfB_{k}\bfA\t_{k}\left(\bfA_{k}\bfx_{k-1}-\bfb_{k}\right) \rvert \, \mathcal{F}_{k-1} \right]  \\
& =  \bfx_{k-1} - \widehat{\bfx} -\bfB\left(\bfx_{k-1}-\bfB^{-1}\mathbb{E}\bfB_{k}\bfA\t_{k}\bfb_{k}\right)  \\
& = \left(\bfI- \bfB\right)\left(\bfx_{k-1}-\widehat{\bfx}\right) = (\mathbb{E}\,\bfB_k/\alpha )\,\left(\bfx_{k-1}-\widehat{\bfx}\right) ,
\end{align*}
and therefore
\[
\mathbb{E} \,\bfx_{k}  - \widehat{\bfx}  = (\mathbb{E}\,\bfB_k/\alpha ) (\mathbb{E}\,\bfx_{k-1}-\widehat{\bfx})
= (\mathbb{E}\, \bfB_k/\alpha )^{k} (\bfx_{0}-\widehat{\bfx}),
 \]
where the last equality comes from the fact that $\bfA_{k}$ are i.i.d. Using Lemma \ref{lem:B}(i)   we get
\[
\norm[2]{\mathbb{E}\bfx_{k} - \widehat{\bfx}}  \leq \norm[2]{\mathbb{E}\,\bfB_k/\alpha}^{k}\norm[2]{\bfx_{0}-\widehat{\bfx}}\to 0.
\]
Thus, $\mathbb{E}\,\bfx_{k} \rightarrow \widehat{\bfx}$ linearly.\\
(ii) Next we show that $\bfx_{k}$ converges linearly to a convergence horizon of $\widehat{\bfx}$. The recursion of $\bfx_k$
leads to
\[
\norm[2]{\bfx_{k} - \widehat{\bfx}}^{2}  =  \norm[2]{\bfx_{k-1}-\widehat{\bfx}}^{2} - 2 \left(\bfx_{k-1} - \widehat{\bfx}\right)\t  \bfB_{
k}\bfA\t_{k}(\bfA_{k}\bfx_{k-1}-\bfb_{k})
 + \norm[2]{\bfB_{k}\bfA\t_{k}(\bfA_{k}\bfx_{k-1}-\bfb_{k})}^2.
 \]
 We find an upper bound for the last term using Lemma \ref{lem:B}(i):
\begin{align*}
\norm[2]{\bfB_{k}\bfA\t_{k}(\bfA_{k}\bfx_{k-1}-\bfb_{k})}^2
& = \norm[2]{\bfB_{k}\bfA\t_{k}\bfA_{k}(\bfx_{k-1}-\widehat{\bfx}) + \bfB_{k}\bfA\t_{k}(\bfA_{k}\widehat{\bfx}-\bfb_{k})}^{2}\\
 & \leq  2\norm[2]{\bfB_{k}\bfA\t_{k}\bfA_{k}(\bfx_{k-1}-\widehat{\bfx})}^{2}+ 2\norm[2]{\bfB_{k}\bfA\t_{k}(\bfA_{k}\widehat{\bfx}-\bfb_{k})}^{2}  \\
 & \leq  2\norm[2]{\bfB_{k}\bfA\t_{k}\bfA_{k}(\bfx_{k-1}-\widehat{\bfx})}^{2}+ 2\alpha^{2}\norm[2]{\bfA\t_{k}(\bfA_{k}\widehat{\bfx}-\bfb_{k})}^{2},    \end{align*}
 and by Lemma \ref{lem:B}(iii)
 \[
 \norm[2]{\bfB_{k}\bfA\t_{k}\bfA_{k}(\bfx_{k-1}-\widehat{\bfx})}^{2}\leq \frac{\alpha\,A_{\max}}{1+\alpha\,A_{\max}}
 (\bfx_{k-1}-\widehat{\bfx})^\top\bfB_k\bfA_k^\top\bfA_k (\bfx_{k-1}-\widehat{\bfx}).
 \]
 The last two bounds yield
\begin{align*} \norm[2]{\bfx_{k} - \widehat{\bfx}}^{2} &  \leq \norm[2]{\bfx_{k-1}-\widehat{\bfx}}^{2} - 2 \left(\bfx_{k-1} - \widehat{\bfx}\right)\t \bfB_{k}\bfA\t_{k}(\bfA_{k}\bfx_{k-1}-\bfb_{k}) \\
 & + \frac{2\alpha A_{\text{max}}}{1+\alpha A_{\text{max}}} \left( \bfx_{k-1} - \widehat{\bfx}\right)\t \bfB_{k}\bfA\t_{k}\bfA_{k}\left(\bfx_{k-1}-\widehat{\bfx}\right) + 2\alpha^2\norm[2]{\bfA\t_{k}(\bfA_{k}\widehat{\bfx}-\bfb_{k})}^{2},
\end{align*}
whose conditional expectation and Lemma \ref{lem:B}(iv) give us
\begin{align*}
\mathbb{E}\left[ \norm[2]{\bfx_{k} - \widehat{\bfx}}^{2} \, \rvert \, \mathcal{F}_{k-1} \, \right]
  \leq & \norm[2]{\bfx_{k-1}-\widehat{\bfx}}^{2} - 2 \left( \bfx_{k-1} - \widehat{\bfx}\right)\t \bfB \left(\bfx_{k-1}-\widehat{\bfx}\right) \\
 & + \frac{2\alpha A_{\text{max}}}{1+\alpha A_{\text{max}}} \left( \bfx_{k-1} -
 \widehat{\bfx}\right)\t \bfB \left(\bfx_{k-1}-\widehat{\bfx}\right)
+  2\alpha^2\mathbb{E}\norm[2]{\bfA\t_{k}(\bfA_{k}\widehat{\bfx}-\bfb_{k})}^{2} \\
	= & \norm[2]{\bfx_{k-1}-\widehat{\bfx}}^{2} - \frac{2}{1+\alpha A_{\text{max}}}\left(\bfx_{k-1} - \widehat{\bfx}\right)\t \bfB \left(\bfx_{k-1}-\widehat{\bfx}\right)\\& +2\alpha^2\,\mathbb{E}\norm[2]{\bfA\t_{k}(\bfA_{k}\widehat{\bfx}-\bfb_{k})}^{2}\\
\leq &
\left(1 - 2c \right)\norm[2]{\bfx_{k-1}-\widehat{\bfx}}^{2}  +2\alpha^{2} \mathbb{E}
\norm[2]{\bfA\t_{k}(\bfA_{k}\widehat{\bfx}-\bfb_{k})}^{2},
\end{align*}
where $c = {\lambda_{\text{min}}\left(\bfB\right)}/{(1+\alpha A_{\text{max}})}.$
Then, the expected squared norm of the error can be bounded using the fact that $0 < 1-2c <1$ by Lemma \ref{lem:B}(iv)
\begin{align*}
\mathbb{E}\norm[2]{\bfx_{k}-\widehat{\bfx}}^{2}& \leq  \left(1 - 2c \right)^{k}\mathbb{E}\norm[2]{\bfx_{0}-\widehat{\bfx}}^{2}
 + 2\alpha^{2}\mathbb{E}\norm[2]{\bfA\t_{k}(\bfA_{k}\widehat{\bfx}-\bfb_{k})}^{2} \, \sum_{i=0}^{k-1}  \left(1 - 2c \right)^{i}
\\ & \leq  \left(1 - 2c\right)^{k}\norm[2]{\bfx_{0}-\widehat{\bfx}}^{2}  + \alpha^{2}c^{-1}\mathbb{E}\norm[2]{\bfA\t_{k}(\bfA_{k}\widehat{\bfx}-\bfb_{k})}^{2}.
\end{align*}
\end{proof}

\begin{proof}[\textbf{Proof of Lemma \ref{thm:bias}}]
By \eqref{eq:xh-xls}, $\widehat{\bfx} - \bfx_{\text{LS}} = \bfB^{-1}\bfC \bfQ_{\bfA}\bfb$, where
\[
\bfC = \mathbb{E}\,\bfB_k\bfA_k^\top\bfW_k^\top.
\]
Since $\bfQ_{\bfA}$ projects onto the orthogonal complement of the column space of $\bfA$, we have
\[
\bfB^{-1}\bfC \bfQ_{\bfA} = (\bfB^{-1}\bfC - (\bfA^\top\bfA)^{-1}\bfA^\top)\bfQ_{\bfA},
\]
and also using Lemma \ref{lem:B}(iii)
\begin{align*}
\|\bfB^{-1}\bfC - (\bfA^\top\bfA)^{-1}\bfA^\top\|_2 &=\|\bfB^{-1}\mathbb{E}\,(\bfB_k-\alpha\bfI)\bfA_k^\top\bfW_k^\top
+ \alpha\bfB^{-1}\mathbb{E}\,\bfA_k\t\bfW_k^\top - (\bfA^\top\bfA)^{-1}\bfA^\top\|_2\\
&\leq \|\bfB^{-1}\|_2\|\mathbb{E}\,\alpha\bfB_k\bfA_k^\top\bfA_k\bfA_k^\top\bfW_k^\top\|_2
+ \|\alpha\beta \bfB^{-1}\bfA^\top - (\bfA^\top\bfA)^{-1}\bfA^\top\|_2\\
&\leq \frac{\alpha^2 A_{\max}\,\|\bfB^{-1}\|_2}{1+\alpha A_{\max}}\,\mathbb{E}\,\|\bfA_k^\top\bfW_k^\top\|_2
+ \|\alpha\beta \bfB^{-1} - (\bfA^\top\bfA)^{-1}\|_2\,\|\bfA\|_2.
\end{align*}
Furthermore,
\begin{equation*}
 \norm[2]{\alpha \beta \bfB^{-1}-\left(\bfA\t\bfA\right)^{-1}} \leq \norm[2]{\bfB^{-1}}\norm[2]{\alpha \beta\bfA\t\bfA- \bfB}\norm[2]{\left(\bfA\t\bfA\right)^{-1}}
\end{equation*}
and
\begin{equation*}
	\norm[2]{\alpha \beta\bfA\t\bfA- \bfB} = \norm[2]{\bbE (\alpha \bfI - \bfB_k)\bfA_k\t \bfA_k} \leq  \frac{\alpha^2 A_{\max}}{1+\alpha A_{\max}}\bbE \norm[2]{\bfA\t_k\bfA_k}.
\end{equation*}
Using the upper bound for $\|\bfB^{-1}\|$ from Lemma \ref{lem:B}(iv) we can now write
\[
\|\bfB^{-1}\bfC - (\bfA^\top\bfA)^{-1}\bfA^\top\|_2 \leq  \frac{\alpha^2 A_{\max}}{1+\alpha A_{\max}}\,\|\bfB^{-1}\|_2\,C
\leq \frac{\alpha A_{\max}}{1+\alpha A_{\max}}\frac{M(1+\alpha A_{\min})}{A_{\min}} \,C
\]
where  $C= \mathbb{E}\,\|\bfA_k^\top\bfW_k\|_2 + \|(\bfA^\top\bfA)^{-1}\|_2\,\|\bfA\|_2\,\mathbb{E}\,\|\bfA_k^\top\bfA_k\|_2$. The upper bound finally follows from
\[
\|\widehat{\bfx} - \bfx_{\text{LS}}\|_2 = \|\bfB^{-1}\bfC \bfQ_{\bfA}\bfb\|_2 \leq
\|\bfB^{-1}\bfC - (\bfA^\top\bfA)^{-1}\bfA^\top\|_2\,\|\bfQ_{\bfA}\bfb\|_2.
\]

\vspace*{-4.5ex}
\[\left.\right.\]
\end{proof}

\section*{References}
\bibliographystyle{plain}
\bibliography{references}

\begin{thebibliography}{10}

\bibitem{andersen2014generalized}
MS~Andersen and PC~Hansen.
\newblock Generalized row-action methods for tomographic imaging.
\newblock {\em Numerical Algorithms}, 67(1):121--144, 2014.

\bibitem{avron2010blendenpik}
H~Avron, P~Maymounkov, and S~Toledo.
\newblock Blendenpik: Supercharging {LAPACK}'s least-squares solver.
\newblock {\em SIAM Journal on Scientific Computing}, 32(3):1217--1236, 2010.

\bibitem{benveniste2012}
A~Benveniste, SS~Wilson, M~Metivier, and P~Priouret.
\newblock {\em Adaptive Algorithms and Stochastic Approximations}.
\newblock Stochastic Modelling and Applied Probability. Springer, New York,
  2012.

\bibitem{bjorck1996numerical}
A~Bj\"{o}rck.
\newblock {\em Numerical Methods for Least Squares Problems}.
\newblock SIAM, 1996.

\bibitem{Bottou1998}
L~Bottou.
\newblock {\em Online Learning in Neural Networks}, chapter 2. Online learning
  and stochastic approximations, pages 9--42.
\newblock Cambridge University Press, 1998.

\bibitem{bottou2004large}
L~Bottou and YL~Cun.
\newblock Large scale online learning.
\newblock In {\em Advances in Neural Information Processing Systems}, pages
  217--224, 2004.

\bibitem{bottou2005line}
L~Bottou and YL~Cun.
\newblock On-line learning for very large data sets.
\newblock {\em Applied Stochastic Models in Business and Industry},
  21(2):137--151, 2005.

\bibitem{bottou2018optimization}
L~Bottou, FE~Curtis, and J~Nocedal.
\newblock Optimization methods for large-scale machine learning.
\newblock {\em SIAM Review}, 60(2):223--311, 2018.

\bibitem{boumal2018heterogeneous}
N~Boumal, T~Bendory, RR~Lederman, and A~Singer.
\newblock Heterogeneous multireference alignment: A single pass approach.
\newblock In {\em 2018 52nd Annual Conference on Information Sciences and
  Systems (CISS)}, pages 1--6. IEEE, 2018.

\bibitem{Byrd2016}
RH~Byrd, SL~Hansen, J~Nocedal, and Y~Singer.
\newblock A stochastic quasi-{N}ewton method for large-scale optimization.
\newblock {\em SIAM Journal on Optimization}, 26(2):1008--1031, 2016.

\bibitem{censor1983strong}
Y~Censor, PB~Eggermont, and D~Gordon.
\newblock Strong underrelaxation in {K}aczmarz's method for inconsistent
  systems.
\newblock {\em Numerische Mathematik}, 41(1):83--92, 1983.

\bibitem{censor2009note}
Y~Censor, GT~Herman, and M~Jiang.
\newblock A note on the behavior of the randomized {K}aczmarz algorithm of
  {S}trohmer and {V}ershynin.
\newblock {\em Journal of Fourier Analysis and Applications}, 15(4):431--436,
  2009.

\bibitem{chen2006regression}
Y~Chen, G~Dong, J~Han, J~Pei, BW~Wah, and J~Wang.
\newblock Regression cubes with lossless compression and aggregation.
\newblock {\em IEEE Transactions on Knowledge and Data Engineering},
  18(12):1585--1599, 2006.

\bibitem{chung2017stochastic}
J~Chung, M~Chung, JT~Slagel, and L~Tenorio.
\newblock Stochastic {N}ewton and quasi-{N}ewton methods for large linear
  least-squares problems.
\newblock {\em arXiv preprint arXiv:1702.07367}, 2017.

\bibitem{chung2006numerical}
J~Chung, E~Haber, and J~Nagy.
\newblock Numerical methods for coupled super-resolution.
\newblock {\em Inverse Problems}, 22(4):1261, 2006.

\bibitem{egidi2006sherman}
N~Egidi and P~Maponi.
\newblock A {S}herman--{M}orrison approach to the solution of linear systems.
\newblock {\em Journal of Computational and Applied Mathematics},
  189(1):703--718, 2006.

\bibitem{eldar2011acceleration}
YC~Eldar and D~Needell.
\newblock Acceleration of randomized {K}aczmarz method via the
  {J}ohnson--{L}indenstrauss lemma.
\newblock {\em Numerical Algorithms}, 58(2):163--177, 2011.

\bibitem{elfving1980block}
T~Elfving.
\newblock Block-iterative methods for consistent and inconsistent linear
  equations.
\newblock {\em Numerische Mathematik}, 35(1):1--12, 1980.

\bibitem{elfving2014semi}
T~Elfving, PC~Hansen, and T~Nikazad.
\newblock Semi-convergence properties of {K}aczmarz's method.
\newblock {\em Inverse Problems}, 30(5):055007, 2014.

\bibitem{feichtinger1992new}
HG~Feichtinger, C~Cenker, M~Mayer, H~Steier, and T~Strohmer.
\newblock New variants of the {POCS} method using affine subspaces of finite
  codimension with applications to irregular sampling.
\newblock In {\em Visual Communications and Image Processing'92}, volume 1818,
  pages 299--311. International Society for Optics and Photonics, 1992.

\bibitem{golub2012matrix}
GH~Golub and CF~Van~Loan.
\newblock {\em Matrix Computations}, volume~3.
\newblock JHU press, 2012.

\bibitem{gordon1970algebraic}
R~Gordon, R~Bender, and GT~Herman.
\newblock Algebraic reconstruction techniques ({ART}) for three-dimensional
  electron microscopy and {X}-ray photography.
\newblock {\em Journal of Theoretical Biology}, 29(3):471--481, 1970.

\bibitem{Gower2016s}
RM~Gower, D~Goldfarb, and P~Richt{\'a}rik.
\newblock Stochastic block {BFGS}: {S}queezing more curvature out of data.
\newblock In {\em International Conference on Machine Learning}, pages
  1869--1878, 2016.

\bibitem{Gower2015}
RM~Gower and P~Richt{\'a}rik.
\newblock Randomized iterative methods for linear systems.
\newblock {\em SIAM Journal on Matrix Analysis and Applications},
  36(4):1660--1690, 2015.

\bibitem{gower2017randomized}
RM~Gower and P~Richt{\'a}rik.
\newblock Randomized quasi-{N}ewton updates are linearly convergent matrix
  inversion algorithms.
\newblock {\em SIAM Journal on Matrix Analysis and Applications},
  38(4):1380--1409, 2017.

\bibitem{halko2011finding}
N~Halko, PG~Martinsson, and JA~Tropp.
\newblock Finding structure with randomness: Probabilistic algorithms for
  constructing approximate matrix decompositions.
\newblock {\em SIAM Review}, 53(2):217--288, 2011.

\bibitem{hamalainen2015tomographic}
K~H{\"a}m{\"a}l{\"a}inen, L~Harhanen, A~Kallonen, A~Kujanp{\"a}{\"a}, E~Niemi,
  and S~Siltanen.
\newblock Tomographic x-ray data of a walnut.
\newblock {\em arXiv preprint arXiv:1502.04064}, 2015.

\bibitem{hand2000data}
DJ~Hand, G~Blunt, MG~Kelly, NM~Adams, et~al.
\newblock Data mining for fun and profit.
\newblock {\em Statistical Science}, 15(2):111--131, 2000.

\bibitem{hanke1990acceleration}
M~Hanke and W~Niethammer.
\newblock On the acceleration of {K}aczmarz's method for inconsistent linear
  systems.
\newblock {\em Linear Algebra and its Applications}, 130:83--98, 1990.

\bibitem{hansen2010discrete}
PC~Hansen.
\newblock {\em Discrete Inverse Problems: Insight and Algorithms}.
\newblock SIAM, 2010.

\bibitem{herman2009fundamentals}
GT~Herman.
\newblock {\em Fundamentals of Computerized Tomography: Image Reconstruction
  from Projections}.
\newblock Springer, 2009.

\bibitem{herman1993algebraic}
GT~Herman and LB~Meyer.
\newblock Algebraic reconstruction techniques can be made computationally
  efficient (positron emission tomography application).
\newblock {\em IEEE Transactions on Medical Imaging}, 12(3):600--609, 1993.

\bibitem{tomobox}
J~Jorgensen.
\newblock Tomobox.
\newblock
  \url{https://www.mathworks.com/matlabcentral/fileexchange/28496-tomobox?s_tid=prof_contriblnk}.
\newblock Accessed: September 2019.

\bibitem{kaczmarz1937angeniherte}
S~Kaczmarz.
\newblock Angen\"{a}herte {A}ufl\"{o}sung von {S}ystemen linearer
  {G}leichungen.
\newblock {\em Bulletin International de l'Acad\'{e}mie Polonaise des Sciences
  et des Lettres. Classe des Sciences Math\'{e}matiques et Naturelles.
  S\'{e}rie A, Sciences Math\'{e}matiques}, 35:335--357, 1937.

\bibitem{kushner2003stochastic}
H~Kushner and GG~Yin.
\newblock {\em Stochastic approximation and recursive algorithms and
  applications}, volume~35.
\newblock Springer Science \& Business Media, 2003.

\bibitem{levin20183d}
E~Levin, T~Bendory, N~Boumal, J~Kileel, and A~Singer.
\newblock {3D} ab initio modeling in cryo-{EM} by autocorrelation analysis.
\newblock In {\em 2018 IEEE 15th International Symposium on Biomedical Imaging
  (ISBI 2018)}, pages 1569--1573. IEEE, 2018.

\bibitem{meng2014lsrn}
X~Meng, MA~Saunders, and MW~Mahoney.
\newblock {LSRN}: A parallel iterative solver for strongly over- or
  underdetermined systems.
\newblock {\em SIAM Journal on Scientific Computing}, 36(2):C95--C118, 2014.

\bibitem{mokhtari2015global}
A~Mokhtari and A~Ribeiro.
\newblock Global convergence of online limited memory {BFGS}.
\newblock {\em The Journal of Machine Learning Research}, 16(1):3151--3181,
  2015.

\bibitem{natterer2001mathematics}
F~Natterer.
\newblock {\em The Mathematics of Computerized Tomography}.
\newblock SIAM, 2001.

\bibitem{needell2010randomized}
D~Needell.
\newblock Randomized {K}aczmarz solver for noisy linear systems.
\newblock {\em BIT Numerical Mathematics}, 50(2):395--403, 2010.

\bibitem{needell2016stochastic}
D~Needell, N~Srebro, and R~Ward.
\newblock Stochastic gradient descent, weighted sampling, and the randomized
  {K}aczmarz algorithm.
\newblock {\em Mathematical Programming}, 155(1-2):549--573, 2016.

\bibitem{Needell2014}
D~Needell and JA~Tropp.
\newblock Paved with good intentions: {A}nalysis of a randomized block
  {K}aczmarz method.
\newblock {\em Linear Algebra and its Applications}, 441:199--221, 2014.

\bibitem{needell2015randomized}
D~Needell, R~Zhao, and A~Zouzias.
\newblock Randomized block {K}aczmarz method with projection for solving least
  squares.
\newblock {\em Linear Algebra and its Applications}, 484:322--343, 2015.

\bibitem{nemirovski2009robust}
A~Nemirovski, A~Juditsky, G~Lan, and A~Shapiro.
\newblock Robust stochastic approximation approach to stochastic programming.
\newblock {\em SIAM Journal on Optimization}, 19(4):1574--1609, 2009.

\bibitem{PaSa82b}
CC~Paige and MA~Saunders.
\newblock {{\def\LSQRb{}}Algorithm 583}, {LSQR:} {S}parse linear equations and
  least-squares problems.
\newblock {\em ACM Transactions on Mathematical Software}, 8(2):195--209, 1982.

\bibitem{paige1982lsqr}
CC~Paige and MA~Saunders.
\newblock {LSQR}: An algorithm for sparse linear equations and sparse least
  squares.
\newblock {\em ACM Transactions on Mathematical Software}, 8(1):43--71, 1982.

\bibitem{parkinson2018achieving}
DY~Parkinson, JI~Pacold, M~Gross, TD~McDougall, C~Jones, J~Bows, I~Hamilton,
  DE~Smiles, S~De~Santis, A~Ratti, et~al.
\newblock Achieving fast high-resolution {3D} imaging by combining synchrotron
  x-ray micro{CT}, advanced algorithms, and high performance data management.
\newblock In {\em Image Sensing Technologies: Materials, Devices, Systems, and
  Applications V}, volume 10656, page 106560S. International Society for Optics
  and Photonics, 2018.

\bibitem{parkinson2017machine}
DY~Parkinson, DM~Pelt, T~Perciano, D~Ushizima, H~Krishnan, HS~Barnard,
  AA~MacDowell, and J~Sethian.
\newblock Machine learning for micro-tomography.
\newblock In {\em Developments in {X}-Ray Tomography XI}, volume 10391, page
  103910J. International Society for Optics and Photonics, 2017.

\bibitem{Rajaraman:2011:MMD:2124405}
A~Rajaraman and JD~Ullman.
\newblock {\em Mining of Massive Datasets}.
\newblock Cambridge University Press, New York, NY, USA, 2011.

\bibitem{schaul2013no}
T~Schaul, S~Zhang, and YL~Cun.
\newblock No more pesky learning rates.
\newblock In {\em International Conference on Machine Learning}, pages
  343--351, 2013.

\bibitem{shapiro2009lectures}
A~Shapiro, D~Dentcheva, and A~Ruszczy{\'n}ski.
\newblock {\em Lectures on Stochastic Programming: Modeling and Theory}.
\newblock SIAM, Philadelphia, 2009.

\bibitem{slagel2019sampled}
JT~Slagel, J~Chung, M~Chung, D~Kozak, and L~Tenorio.
\newblock Sampled {T}ikhonov regularization for large linear inverse problems.
\newblock {\em Inverse Problems}, 35(114008), 2019.

\bibitem{strohmer2009comments}
T~Strohmer and R~Vershynin.
\newblock Comments on the randomized {K}aczmarz method.
\newblock {\em Journal of Fourier Analysis and Applications}, 15(4):437--440,
  2009.

\bibitem{strohmer2009randomized}
T~Strohmer and R~Vershynin.
\newblock A randomized {K}aczmarz algorithm with exponential convergence.
\newblock {\em Journal of Fourier Analysis and Applications}, 15(2):262--278,
  2009.

\bibitem{tanabe1971projection}
K~Tanabe.
\newblock Projection method for solving a singular system of linear equations
  and its applications.
\newblock {\em Numerische Mathematik}, 17(3):203--214, 1971.

\bibitem{whitney1967two}
TM~Whitney and RK~Meany.
\newblock Two algorithms related to the method of steepest descent.
\newblock {\em SIAM Journal on Numerical Analysis}, 4(1):109--118, 1967.

\bibitem{zeiler2012adadelta}
MD~Zeiler.
\newblock {ADADELTA}: an adaptive learning rate method.
\newblock {\em arXiv preprint arXiv:1212.5701}, 2012.

\bibitem{zeng2014incremental}
XQ~Zeng and GZ~Li.
\newblock Incremental partial least squares analysis of big streaming data.
\newblock {\em Pattern Recognition}, 47(11):3726--3735, 2014.

\bibitem{zouzias2013randomized}
A~Zouzias and NM~Freris.
\newblock Randomized extended {K}aczmarz for solving least squares.
\newblock {\em SIAM Journal on Matrix Analysis and Applications},
  34(2):773--793, 2013.

\end{thebibliography}

\end{document}